\documentclass[a4paper,11pt]{amsart}
\usepackage[T1]{fontenc}
\usepackage[english]{babel}
\usepackage[cp1252]{inputenc}
\usepackage{amsthm}
\usepackage{amsmath}
\usepackage{amsfonts}
\usepackage{tikz-cd}
\usepackage{pgfplots}
 \usepackage{microtype}
\usepackage[margin=0.8in]{geometry}

\usepackage{xcolor}
 \pgfplotsset{compat=1.18}
\usepgfplotslibrary{fillbetween}

\usepackage{amssymb}
\usepackage{hyperref}
\usepackage{color}
\usepackage{caption}
\usepackage{indentfirst}
\usepackage{amssymb}
\usepackage{eufrak}
\usepackage{mathrsfs}
\usepackage{xypic}

 \usepackage{booktabs,array,graphicx,longtable}
\usepackage{array,booktabs,tabularx}

\usepackage{booktabs}
\usepackage{listings}
\theoremstyle{plain}   

\usepackage{booktabs}
\usepackage{graphicx}
\usepackage{array}

\definecolor{amoebadarkblue}{RGB}{0,63,127}
\definecolor{contourred}{RGB}{224,0,0}

\newif\ifhighdensityplot
\highdensityplotfalse

\definecolor{ContourColor}{RGB}{100,0,10}

\newtheorem{question}{Question}

\newcommand{\R}{\mathbb R}
\newcommand{\C}{\mathbb C}

\newcommand{\Z}{\mathbb Z}

\newcommand{\Log}{\mathrm{Log}\,}

\newcommand{\Q}{\mathbb Q}

 \newcommand{\Logmap}{\operatorname{Log}}

\newcommand{\di}{\displaystyle}

\newcommand{\Crit}{\operatorname{Crit}}

\newcommand{\Area}{\operatorname{Area}}

\newcommand{\supp}{\operatorname{supp}}

\newcommand{\Newt}{\operatorname{Newt}}

\newcommand{\Res}{\operatorname{Res}}

\usepackage{array}
\usepackage{tabularx}

 \usepackage{caption}
 
\newcommand\restr[2]{{% we make the whole thing an ordinary symbol
  \left.\kern-\nulldelimiterspace % automatically resize the bar with \right
  #1 % the function
  \right|_{#2} % this is the delimiter
  }}

\newtheorem{remark}{Remark}[section]
\newtheorem*{mtheorem*}{Main Theorem}
\newtheorem{theorem}{Theorem}[section]
\newtheorem{definition}{Definition}[section]
\newtheorem{proposition}{Proposition}[section]
\newtheorem{corollary}{Corollary}[section]
\newtheorem{lemma}{Lemma}[section]
\begin{document}
\title{Beyond Harnack Rigidity}  
\author{Mounir Nisse}
\date{}

 \email{\href{mailto:mounir.nisse@gmail.com}{mounir.nisse@gmail.com}, {mounir.nisse@xmu.edu.my}}

\urladdr{www.math.tamu.edu/\~{}nisse}
 
\thanks{This research is supported in part by Xiamen University Malaysia Research Fund (Grant no. XMUMRF/ 2024-C5/IMAT/0013).}

 

\maketitle

\begin{abstract}
The boundary of a plane amoeba is always contained in its contour, and equality is a characteristic feature of simple Harnack curves. We show that the converse fails, even under strong smoothness and nondegeneracy assumptions. For every
   two-dimensional lattice polygon,  except unimodular triangles,
   we construct a smooth Newton-nondegenerate curve with smooth logarithmic critical locus and smooth embedded contour satisfying $\mathcal C(\mathscr A_f)=\partial\mathscr A_f$, although the curve is not Harnack. We also provide explicit primitive and nonprimitive families that are not torus-equivalent to simple Harnack curves. A concrete primitive example is certified by exact elimination and Sturm root counting. These results disprove contour--boundary rigidity and show that the contour as a set does not detect the real structure or the multiplicity of coincident critical sheets, thereby refining the compensation problem proposed by Lang, Shapiro, and Shustin.
\end{abstract}

\section{Introduction}

Amoebas transform complex algebraic curves into real geometric objects whose shape reflects algebraic, analytic, and combinatorial information.  If $C\subset(\C^*)^2$ is a curve, its amoeba is the image of $C$ under the logarithmic map $\Log(z,w)=(\log|z|,\log|w|)$.  The Newton polygon governs the asymptotic directions of this image, while the topology of the complement and the geometry of the boundary retain traces of the defining equation.  This interaction has made amoebas a natural meeting point of toric geometry, real algebraic geometry, and tropical geometry.

Simple Harnack curves provide the most ordered examples.  Mikhalkin introduced them through their extremal position in toric surfaces and established the close relation between their real loci and amoeba boundaries \cite{Mikhalkin2000}.  Their moduli and deformation theory reveal the persistence of this exceptional geometry \cite{Olarte2021}, while generalizations based on the total reality of the logarithmic Gauss map show that the Harnack phenomenon belongs to a broader real-geometric picture \cite{Lang2024}.  The extremal character of these curves is also measured by area: Mikhalkin and Rullg{\aa}rd proved that the area of a plane amoeba with Newton polygon $\Delta$ is bounded above by $\pi^2\operatorname{Area}(\Delta)$ and that equality is governed by Harnack geometry \cite{MikhalkinRullgard2001}.

The contour of an amoeba records the critical values of the logarithmic map restricted to the curve.  The boundary is contained in the contour under the usual regularity assumptions, but the two sets need not agree.  A general contour may contain critical branches inside the amoeba and may develop crossings, cusps, tangencies, or other singularities even when the algebraic curve is smooth.  The contour therefore records both the visible folds forming the boundary and additional folds hidden in the interior.

Lang, Shapiro, and Shustin studied this geometry through the real degree of the contour, measured by its maximal number of transverse intersections with a line \cite{LangShapiroShustin2021}.  Their work emphasizes that contours are semi-analytic but need not be analytic in general.  They raise the challenging problem of finding simple nondegeneracy conditions that guarantee analyticity.  The present work is motivated by this question and examines a particularly regular regime in which the logarithmic critical locus is smooth and the contour is a smooth embedded real-analytic curve.

The first main result shows that such regularity is not restricted to special Newton polygons or to the standard real Harnack setting.  For every two-dimensional lattice polygon $\Delta$, we prove the existence of a Laurent polynomial $f\in\mathcal S_\Delta$ for which the curve $V_f$ is smooth and Newton nondegenerate, the logarithmic critical locus is smooth, and the contour is a smooth embedded real-analytic curve satisfying $\mathcal C(\mathscr A_f)=\partial\mathscr A_f$.  Moreover, $f$ may be chosen nonreal even after multiplication by a nonzero constant.  Consequently, $V_f$ need not be a Harnack curve with respect to the standard real structure.

This universal existence theorem gives a broad answer to the regularity aspect of the problem.  It shows that every lattice polygon supports curves whose complete critical-value image is as regular as one could naturally request and contains no branch in the interior of the amoeba.  Neither the combinatorial type of the polygon nor standard reality prevents this phenomenon.  Smooth contour--boundary equality is therefore a flexible feature across the entire class of plane Newton polygons.

The second main result makes this flexibility explicit through an infinite family.  For every integer $m>1$, consider
$$
f_m(z,w)=\left(z^m-\frac3{10}\right)w-(z^m+3)(z^m+10).
$$
Its Newton polygon is $\Delta_m=\operatorname{conv}\{(0,0),(2m,0),(m,1),(0,1)\}$.  The associated curve is smooth, its contour is smooth, and the contour coincides with the amoeba boundary.  Nevertheless, the curve is not torus-equivalent to any simple Harnack curve.  Thus the failure of Harnack rigidity is not confined to one isolated example: it occurs in a concrete family of increasing degree and with an immediately visible algebraic form.

The family $f_m$ is important for two reasons.  It gives a transparent source of non-Harnack curves with perfectly regular amoeba boundaries, and it separates torus equivalence from the weaker question of standard reality.  A curve may fail to be Harnack not simply because its displayed equation has nonreal coefficients, but because no allowable torus transformation can place it in the simple Harnack class.  The examples therefore expose a geometric obstruction rather than a superficial choice of coefficients.

The exponent structure of $f_m$ also points toward a sharper rigidity question.  Since the powers of $z$ occur in multiples of $m$, the family is naturally related to finite monomial coverings.  It remains reasonable to ask whether contour--boundary equality might force Harnack geometry after this covering mechanism is excluded.  This leads to the primitive-support problem, in which the exponent differences are required to generate the full lattice $\Z^2$.

The third main result answers that stronger question negatively.  We construct the explicit primitive curve defined by
$$
f_0(z,w)=1+z+\left(\frac{1}{100}-\frac{i}{5000}\right)z^2+w+
\left(\frac{1}{20}-\frac{7i}{200}\right)zw.
$$
The curve is irreducible and smooth, its toric compactification is smooth and Newton nondegenerate, and its exponent differences generate $\Z^2$.  Its logarithmic critical locus and contour are smooth, and the complete contour is exactly the boundary of the amoeba.  Yet the curve is not torus-equivalent to a real curve and hence is not torus-equivalent to a smooth simple Harnack curve.

Together, these results give a progressively stronger picture.  The universal theorem establishes smooth contour--boundary equality for every lattice polygon, even outside the standard real setting.  The family $f_m$ supplies infinitely many explicit curves that are not torus-equivalent to simple Harnack curves.  The primitive example then shows that the phenomenon survives after finite monomial pullbacks have been excluded.  The failure of Harnack rigidity is therefore neither polygon-specific, nor isolated, nor solely a consequence of nonprimitive support.

The relation with the Lang--Shapiro--Shustin question is especially revealing.  Their challenge asks when the critical-value image becomes analytic \cite{LangShapiroShustin2021}.  Our existence theorem shows that smooth embedded real-analytic contours equal to amoeba boundaries occur for every lattice polygon.  The explicit family and the primitive counterexample then show that this analytic regularity does not determine the real or Harnack nature of the underlying curve.  Regularity of the contour and rigidity of the curve are thus distinct problems.

This distinction reflects a basic loss of information under the logarithmic map.  Contour--boundary equality describes where the critical values occur, but it does not retain the phases of the coefficients, the identity of the critical lifts, or the multiplicities with which the complex curve covers the amoeba.  These missing data are essential to the Harnack property and to maximal area.  A contour can therefore be smooth, embedded, and entirely exposed on the boundary while the curve behind it remains fundamentally non-Harnack.

The area theorem of Mikhalkin and Rullg{\aa}rd makes this separation particularly clear \cite{MikhalkinRullgard2001}.  Maximal area controls more than the location of critical values; it captures the global organization of logarithmic sheets.  Our examples show that a perfectly regular boundary does not force this organization.  Contour--boundary equality should consequently not be treated as a substitute for maximal area or total reality.

An essential feature of the present work is that the completeness of the contours is rigorously established.  Numerical drawings are indispensable for discovering examples, but they cannot rule out a small interior critical arc or an isolated loss of regularity.  The proofs control the complete logarithmic image, including the noncompact ends, and certify that every critical value belongs to the boundary.  In the primitive example, this verification is exact and global.

The results suggest that any successful replacement for contour--boundary rigidity must include information invisible in the contour as a set.  Possible additional hypotheses involve maximal area, total reality of the logarithmic Gauss map, multiplicity conditions on regular logarithmic fibers, or phase constraints recovering a real form.  Simple Harnack curves remain the model in which real topology, contour, boundary, multiplicity, and area agree \cite{Mikhalkin2000,MikhalkinRullgard2001}; the present work shows that smoothness and contour--boundary equality alone recover only part of that structure.

In this way, the paper moves from a general existence theorem to an explicit infinite family and finally to a primitive counterexample.  The three results together clarify the scope of analytic contour regularity and locate the precise limitation of a natural Harnack characterization.  They also provide a concrete foundation for future multiplicity-sensitive and phase-sensitive approaches to amoeba geometry.

%%%%%%%%%%%%%%%%%%%%%%%%%%%%%%%%%%%%%%%%%%%%%%%%%%%%%%%%%%%%%%%%%%%%%%%

\noindent{\it Acknowledgements.} The author is deeply grateful to Boris Shapiro for kindly sharing the joint paper with Lionel Lang and Eugenii Shustin \cite{LangShapiroShustin2021}; the present work was inspired by the challenging question raised in that paper.
Also,     I am gratefully acknowledges the use of Python, NumPy, SciPy, Matplotlib, Pillow, and LaTeX with TikZ and PGFPlots for the exact computations, numerical experiments, and figures presented in this work.

%%%%%%%%%%%%%%%%%%%%%%%%%%%%%%%%%%%%%%%%%%%%%%%%%%%%%%%%%%%%%%%%%%%%%%%

\section{Conditions for Smoothness of the Contour of a Plane Amoeba}

Let $\Delta\subset\mathbb R^2$ be a two-dimensional lattice polygon and let
$\mathcal S_\Delta$ denote the set of Laurent polynomials
$f(z,w)=\sum_{\alpha\in\Delta\cap\mathbb Z^2}c_\alpha z^{\alpha_1}w^{\alpha_2}$
with $\Newt(f)=\Delta$.  Put
$V_f=\{(z,w)\in(\mathbb C^*)^2:f(z,w)=0\}$ and define
$\Log(z,w)=(\log|z|,\log|w|)$.  The amoeba, logarithmic critical locus,
and contour are
$\mathscr A_f=\Log(V_f)$,
$S_f=\Crit(\Log|_{V_f})$, and
$\mathcal C(\mathscr A_f)=\Log(S_f)$.
Assume throughout that $V_f$ is smooth, that its closure in the toric surface
$X_\Delta$ is Newton nondegenerate and transverse to the toric boundary, and
that $S_f$ is a smooth real one-dimensional submanifold of $V_f$.  These
hypotheses do not by themselves imply that the contour is smooth.  A smooth
critical curve can be mapped by $\Log$ to a curve with a cusp, and two
distinct smooth critical branches can have the same logarithmic value and
produce a node or a tangential intersection.  Smoothness of the source and
smoothness of its image are separate questions.
The logarithmic Gauss map is
$$
\gamma_f:V_f\longrightarrow\mathbb P^1_{\mathbb C},
\qquad
\gamma_f(z,w)=[zf_z(z,w):wf_w(z,w)].
$$
The standard critical-point criterion gives
$S_f=\gamma_f^{-1}(\mathbb P^1_{\mathbb R})$.  Thus the assumed smoothness
of $S_f$ is guaranteed, for example, when $\gamma_f$ is transverse to
$\mathbb P^1_{\mathbb R}$.  In this context transversality means that, for
every $p\in V_f$ satisfying $\gamma_f(p)\in\mathbb P^1_{\mathbb R}$, the
real differential of $\gamma_f$ maps $T_pV_f$ onto a direction which is not
contained in the tangent line of $\mathbb P^1_{\mathbb R}$.  Equivalently,
on an affine Gauss chart with
$h=zf_z/(wf_w)$, the real function $\rho=\operatorname{Im}h$ satisfies
$d(\rho|_{V_f})_p\ne0$ at every point of $S_f$.  This condition makes
$S_f=\{\rho=0\}$ a smooth real curve, but it says nothing yet about the
restriction of $\Log$ to this curve.

The exact local condition is that
$\ell_f=\Log|_{S_f}:S_f\to\mathbb R^2$ have rank one everywhere.  Since
$S_f$ is one-dimensional, this says that for every $p\in S_f$ and every
nonzero tangent vector $\tau\in T_pS_f$, one has
$$
d\ell_f{}_p(\tau)
=\left(
\operatorname{Re}\frac{\tau_z}{z},
\operatorname{Re}\frac{\tau_w}{w}
\right)\ne(0,0).
$$
When this condition holds, the constant-rank theorem shows that the image of
each sufficiently small critical branch is a smooth immersed real-analytic
arc.  Its failure is the ramification mechanism for a singular contour.  A
generic first-order failure produces a cusp, while higher vanishing can
produce a higher-order cusp or another higher-order image singularity.

\noindent There is a useful coordinate form of the ramification condition.  Choose
local real coordinates $(u,v)$ on $V_f$ and write
$X=\log|z|$, $Y=\log|w|$.  Since $S_f=\{\rho=0\}$ and $d\rho\ne0$, its
tangent line is $\ker d\rho$.  The map $\ell_f$ is ramified at $p$ exactly
when both $dX$ and $dY$ vanish on this kernel.  Equivalently,
$$
\det\frac{\partial(\rho,X)}{\partial(u,v)}(p)=0
\quad\hbox{and}\quad
\det\frac{\partial(\rho,Y)}{\partial(u,v)}(p)=0.
$$
The absence of simultaneous solutions of these two equations on $S_f$ is
the complete no-ramification condition.  It is invariant under the chosen
Gauss chart and the chosen coordinates on $V_f$.

No ramification is necessary for the critical-value parametrization to be
an immersion, but it is not sufficient for the contour to be an embedded
smooth curve.  One must also control distinct critical points with the same
logarithmic image.  Such points $p=(z,w)$ and $q=(Z,W)$ satisfy
$$
p,q\in S_f,
\qquad
|z|=|Z|,
\qquad
|w|=|W|.
$$
If $p\ne q$ and the two image tangent lines are distinct, their common value
is a transverse node.  If their tangent lines coincide but the image germs
are different, the contour has a tangency, a tacnode, or a higher-order
contact.  If three or more distinct critical branches have the same image,
one obtains a multiple point unless all their image germs coincide as one
smooth germ.

\vspace{0.1cm}

\noindent This leads to the exact necessary-and-sufficient criterion.  Under the
standing hypotheses, the contour $\mathcal C(\mathscr A_f)$ is a smooth
embedded one-dimensional real-analytic submanifold of $\mathbb R^2$ if and
only if $\ell_f$ has rank one at every point of $S_f$ and, for every
$x\in\ell_f(S_f)$, all local image germs
$\ell_f(U_p)$, where $p$ ranges over $\ell_f^{-1}(x)$ and $U_p$ is a small
neighborhood of $p$ in $S_f$, coincide with a single embedded
real-analytic arc through $x$.  The first clause excludes ramification.  The
second excludes every crossing, tangency, and multiple image germ, while
allowing the harmless possibility that several source branches cover the
same smooth image arc.

\vspace{0.1cm}

\noindent  The necessity is immediate.  If the contour is a smooth embedded curve, a
local parametrization of any contributing branch cannot have zero velocity,
and all contributing image germs at a fixed contour point must be the unique
germ of that smooth curve.  For sufficiency, the rank-one condition and the
constant-rank theorem make every local image an embedded analytic arc.  The
germ-coincidence condition says that the union of these arcs has exactly one
local branch at each image point.  Therefore their union is locally a
one-dimensional embedded real-analytic submanifold, which is precisely the
smoothness of the contour.

\vspace{0.1cm}

\noindent  A simpler and very useful sufficient criterion is obtained by excluding
all multiple lifts.  If $\ell_f$ is an immersion and is injective, then the
contour is smooth and embedded.  Indeed, $\Log:V_f\to\mathbb R^2$ is
proper: the inverse image of a compact set has all coordinate moduli bounded
above and bounded away from zero, its arguments lie in a compact torus, and
the equation $f=0$ defines a closed subset.  The restriction to the closed
critical locus is therefore proper.  A proper injective immersion is an
embedding.  Consequently, the two readily verifiable conditions
$$
d(\Log|_{S_f})_p\ne0
\quad\hbox{for every }p\in S_f
$$
and
$$
p,q\in S_f,
\quad
\Log(p)=\Log(q)
\quad\Longrightarrow\quad
p=q
$$
imply that $\mathcal C(\mathscr A_f)$ is a smooth embedded contour.  This
criterion is stronger than necessary because a finite covering of a smooth
contour arc can have several distinct lifts without making the image
singular.

\vspace{0.1cm}

\noindent For computation, the injectivity condition is expressed by a saturated
two-lift system.  Introduce two copies
$p=(z,w)$ and $q=(Z,W)$ of the critical equations.  After replacing complex
conjugates by independent real coordinates, impose
$|z|^2-|Z|^2=0$ and $|w|^2-|W|^2=0$.  The diagonal equations
$z-Z=w-W=0$, together with their real-coordinate versions, define the
automatic component $p=q$.  Saturating by the diagonal ideal removes
components supported entirely on this automatic solution.  Saturation by
$zwZW$ and, in a toric compactification, by the relevant toric irrelevant
ideals removes coordinate and boundary components which do not represent
finite points of $(\mathbb C^*)^2$.  If the resulting real off-diagonal
scheme has no solution, then distinct critical lifts never share a
logarithmic value, and the no-ramification condition proves smoothness of the
contour.

If the saturated two-lift system does have solutions, the contour need not
be singular.  At each solution one must compare the image germs.  Let
$\tau_p$ and $\tau_q$ be nonzero tangent vectors to $S_f$ at the two lifts,
and put $v_p=d\Log_p(\tau_p)$ and
$v_q=d\Log_q(\tau_q)$.  A nonzero determinant
$$
\det(v_p,v_q)\ne0
$$
establishes a transverse node.  If the determinant is zero, higher jets of
$\ell_f$ along the two critical branches must be compared.  Different
second or higher jets certify tangentially distinct germs.  Only when the
entire analytic image germs agree does the multiple lift remain compatible
with a smooth contour.  For algebraic curves, equality of germs can be
tested by eliminating the local parameters or by proving that the two
branches are related by a local deck transformation preserving $\Log$.

The assumptions concerning the toric compactification control the behavior
of the unbounded ends.  Newton nondegeneracy implies that the closure of
$V_f$ meets every toric divisor smoothly and with the multiplicities
prescribed by the corresponding side of $\Delta$.  Transversality to the
toric boundary gives regular logarithmic ends with directions normal to the
sides of $\Delta$.  These conditions prevent a singularity of the algebraic
curve from being hidden at the toric boundary.  They do not, by themselves,
prevent two finite contour branches from crossing, and they do not replace
the ramification or two-lift conditions.  If one also wants a smooth
compactified contour at infinity, one must require distinct end germs with
the same asymptotic direction either to have separated offsets or to agree
as one compactified germ.  For smoothness as a subset of the finite plane
$\mathbb R^2$, ends escaping to infinity create no additional point to
check.

%%%%%%%%%%%%%%%%%%%%%%%%%%%%%%%%%%%%%%%%%%%%%%%%%%%%%%%%%%%%%%%%%%%%%%%%%%%%%

\section{A Smoothness Criterion for the Contour of a Plane Amoeba}
 
Let $\Delta\subset\mathbb R^2$ be a two-dimensional lattice polygon and let
$f$ be a Laurent polynomial with $\Newt(f)=\Delta$.  
Assume that $V_f$ is smooth, that its toric compactification in $X_\Delta$
is Newton nondegenerate and transverse to the toric boundary, and that
$S_f$ is a smooth real one-dimensional submanifold.  The following theorem
separates smoothness of the critical locus from smoothness of its image.

\begin{theorem}% 
Let $\ell_f=\Logmap|_{S_f}:S_f\to\mathbb R^2$.  Suppose that
$d\ell_f{}_p$ has rank one for every $p\in S_f$.  Suppose moreover that, for
every $x\in\ell_f(S_f)$ and every pair
$p,q\in\ell_f^{-1}(x)$, the local image germs of $\ell_f$ at $p$ and $q$
coincide as one embedded real-analytic curve germ through $x$.  Then
$\mathcal C(\mathscr A_f)$ is a smooth embedded real-analytic curve in
$\mathbb R^2$.

Conversely, if $\mathcal C(\mathscr A_f)$ is a smooth embedded
real-analytic curve and $\ell_f$ is an immersion, then all local image germs
above every contour point coincide with the unique germ of
$\mathcal C(\mathscr A_f)$ at that point.  Thus, within the class in which
$\ell_f$ is unramified, the two stated conditions are necessary and
sufficient.
\end{theorem}

\begin{proof}
Fix $p\in S_f$ and write $x=\ell_f(p)$.  Since $S_f$ is a smooth
one-dimensional real-analytic manifold and $d\ell_f{}_p$ has rank one, the
real-analytic constant-rank theorem gives an open interval $U_p\subset S_f$
containing $p$ such that $\ell_f(U_p)$ is an embedded real-analytic arc
through $x$.  In suitable real-analytic coordinates $t$ on $U_p$ and
$(X,Y)$ near $x$, the map has the form $t\mapsto(t,0)$.  In particular, the
image has neither a cusp nor a higher-order cusp at $x$.

The preceding argument applies to every point of the fiber
$\ell_f^{-1}(x)$.  The second hypothesis says that all arcs obtained in this
way determine the same embedded germ at $x$.  Therefore, after making the
neighborhood of $x$ smaller if necessary, the union of the local images of
all critical branches meeting that neighborhood is a single embedded
real-analytic arc.  No two distinct image germs cross transversely, meet
tangentially, or form a multiple point.

We justify that no additional critical branch can accumulate at $x$ without
being represented by a point of the fiber.  The map
$\Logmap:V_f\to\mathbb R^2$ is proper.  Indeed, if $K\subset\mathbb R^2$ is
compact, then the two coordinate moduli of every point of
$\Logmap^{-1}(K)$ are bounded above and bounded away from zero.  The argument
coordinates lie in the compact torus $(S^1)^2$, and $V_f$ is closed in
$(\mathbb C^*)^2$.  Hence $\Logmap^{-1}(K)\cap V_f$ is compact.  The critical
locus $S_f$ is closed in $V_f$, so the restricted map $\ell_f$ is proper as
well.

Choose a compact closed disk $D$ centered at $x$.  Properness implies that
$\ell_f^{-1}(D)$ is compact.  If critical branches with images meeting every
neighborhood of $x$ were not accounted for by the fiber, one could choose a
sequence $p_n\in S_f$ whose images converge to $x$.  Compactness gives a
convergent subsequence $p_{n_j}\to p_\infty\in S_f$.  Continuity yields
$\ell_f(p_\infty)=x$, so $p_\infty$ belongs to the fiber and its local image
is one of the germs already considered.  This contradiction shows that the
local contour is exactly the single embedded arc supplied by the hypotheses.

Since the argument holds at every $x\in\ell_f(S_f)$, the contour is locally
a one-dimensional embedded real-analytic submanifold of $\mathbb R^2$ at
each of its points.  Hence $\mathcal C(\mathscr A_f)$ is a smooth embedded
real-analytic curve.

For the converse, assume that the contour is a smooth embedded
real-analytic curve and that $\ell_f$ is an immersion.  At a contour point
$x$, there is exactly one embedded contour germ.  For each
$p\in\ell_f^{-1}(x)$, the constant-rank theorem makes the local image of
$\ell_f$ at $p$ an embedded real-analytic arc contained in the contour.  A
one-dimensional embedded submanifold cannot contain a different
one-dimensional embedded germ through the same point.  Thus every local
image germ equals the unique contour germ at $x$.  This proves the converse
and completes the proof.
\end{proof}

\begin{corollary}% 
Under the hypotheses of the theorem, suppose that
$\ell_f=\Logmap|_{S_f}$ is an injective immersion.  Then
$\mathcal C(\mathscr A_f)$ is a smooth embedded real-analytic curve.
\end{corollary}

\begin{proof}
Since $\ell_f$ is an immersion, the constant-rank theorem makes every
local image an embedded real-analytic arc.  Injectivity implies that every
fiber contains one point, so the image-germ compatibility condition in the
theorem is automatic.  The theorem therefore proves the assertion.
Equivalently, properness of $\ell_f$, proved above, shows directly that a
proper injective immersion is an embedding.  Its image is the contour, so
the contour is a smooth embedded real-analytic curve.
\end{proof}

\begin{remark}% 
The injectivity hypothesis can be tested algebraically.  Introduce two
critical points $p=(z,w)$ and $q=(Z,W)$, impose the two copies of the curve
and critical equations together with $|z|^2=|Z|^2$ and
$|w|^2=|W|^2$, and saturate by the diagonal ideal defining $p=q$.  One must
also saturate by the torus-coordinate factors and by the appropriate toric
boundary ideals.  If the resulting real off-diagonal system has no solution,
then distinct critical points never have the same logarithmic value.  If the
ramification system also has no solution, the corollary applies.

If the off-diagonal system has solutions, smoothness is still possible, but
the corresponding local image germs must be compared.  Distinct tangent
lines produce a node.  Equal tangent lines with different higher jets
produce a tangency or a higher-order contact.  Multiple lifts are compatible with smoothness
only when all their complete analytic image germs coincide as the same
embedded arc.
\end{remark}

%%%%%%%%%%%%%%%%%%%%%%%%%%%%%%%%%%%%%%%%%%%%%%%%%%%%%%%%%%%%%%%%%%%%%%%%

\section{Boundary-Only Amoeba Contours Beyond the Simple Harnack Class}

A smooth curve which is not torus-equivalent to
any simple Harnack curve can satisfy
$\mathcal C(\mathscr{A}_f)=\partial\mathscr A_f$.  Such examples can be  obtained by pulling
back a smooth simple Harnack curve through a finite monomial covering of
degree greater than one.  This construction is fundamentally different from
a unitary torus translation: it increases the cardinality of every regular
logarithmic fiber, and that cardinality cannot be changed by a torus
automorphism.

Let $H=V(h)\subset(\mathbb C^*)^2$ be a smooth simple Harnack curve, and let
$A\in\operatorname{Mat}_{2\times2}(\mathbb Z)$ satisfy
$N=|\det A|>1$.  The associated monomial map is
$$
\pi_A(z,w)=
\bigl(z^{a_{11}}w^{a_{12}},z^{a_{21}}w^{a_{22}}\bigr).
$$
It is an unramified covering of $(\mathbb C^*)^2$ of degree $N$.  Put
$C=\pi_A^{-1}(H)$.  When this inverse image is connected, it is a smooth
irreducible curve.  Smoothness follows because $d\pi_A$ is invertible at
every point of the algebraic torus and $H$ is smooth.

If $L_A:\mathbb R^2\to\mathbb R^2$ denotes multiplication by $A$, then
$$
\Logmap\circ\pi_A=L_A\circ\Logmap.
$$
Since $L_A$ is an invertible real linear map, this identity implies
$$
\mathscr A_C=L_A^{-1}(\mathscr A_H)
\quad\hbox{and}\quad
\partial\mathscr A_C=L_A^{-1}(\partial\mathscr A_H).
$$
The same identity determines the critical loci.  At $p\in C$, differentiation
gives
$$
d\Logmap_{\pi_A(p)}\circ d\pi_A{}_p
=L_A\circ d\Logmap_p.
$$
Both $d\pi_A{}_p$ and $L_A$ are invertible.  Hence
$d\Logmap_p$ drops rank on $T_pC$ if and only if
$d\Logmap_{\pi_A(p)}$ drops rank on $T_{\pi_A(p)}H$.  Therefore,
$$
\Crit(\Logmap|_C)=
\pi_A^{-1}\bigl(\Crit(\Logmap|_H)\bigr)
$$
and
$$
  \mathcal C(\mathscr{A}_C)  =L_A^{-1}(   \mathcal C(\mathscr{A}_H)).
$$
Since $H$ is simple Harnack,
$ \mathcal C(\mathscr{A}_H)=\partial\mathscr A_H$.  Combining the preceding equalities
gives the exact identity
$$
 \mathcal C(\mathscr{A}_C)
=L_A^{-1}( \mathcal C(\mathscr{A}_H))
=L_A^{-1}(\partial\mathscr A_H)
=\partial\mathscr A_C.
$$
If the boundary of the amoeba of $H$ is smooth, then the boundary and contour
of $C$ are smooth as well, because $L_A^{-1}$ is a linear diffeomorphism.

The pullback curve is not torus-equivalent to a simple Harnack curve.  At a
regular interior point $y\in\mathscr A_H$, the simple Harnack property gives
exactly two points of $H$ over $y$.  If
$x=L_A^{-1}y$, every one of these two points has exactly $N$ inverse images
under $\pi_A$, all having logarithmic image $x$.  Therefore
$$
\#\bigl(\Logmap|_C\bigr)^{-1}(x)=2N>2.
$$
On the other hand, the logarithmic map of a simple Harnack curve is at most
two-to-one.  A torus automorphism has a unimodular exponent matrix and is
biholomorphic, so it preserves the generic cardinality of logarithmic
fibers, up to an invertible affine transformation of $\mathbb R^2$.
Therefore no torus automorphism can transform $C$ into a simple Harnack
curve.

There is also an explicit nontriangular and nonparallelogram family.  Consider
$$
h(u,w)=\left(u-\frac3{10}\right)w-(u+3)(u+10).
$$
Its Newton polygon is
$\Delta_0=\operatorname{conv}\{(0,0),(2,0),(1,1),(0,1)\}$, a lattice
trapezoid which is neither a triangle nor a parallelogram.  The curve has the
rational parametrization
$$
u=t,
\qquad
w=\frac{(t+3)(t+10)}{t-3/10}.
$$
The zeros $-10,-3$, the pole $3/10$, and the zero and pole of the coordinate
$u=t$ occur on $\mathbb RP^1$ in the cyclic arrangement of the standard
rational simple Harnack parametrization.  Equivalently, if $t=re^{i\theta}$,
direct differentiation gives
$$
\frac{\partial}{\partial\theta}
\log\left|\frac{(t+3)(t+10)}{t-3/10}\right|
=r\sin\theta\left(
-\frac3{|t+3|^2}-\frac{10}{|t+10|^2}
-\frac{3/10}{|t-3/10|^2}\right).
$$
The factor in parentheses is strictly negative at every regular point.
There is therefore precisely one maximum and one minimum on every regular
circle $|t|=r$.  Hence the logarithmic map
is at most two-to-one, which is the amoeba-map characterization of a rational
simple Harnack curve.  Smoothness in the torus can also be checked directly:
the simultaneous equations
$h=h_u=h_w=0$ have no solution in $(\mathbb C^*)^2$.

For an integer $m>1$, pull this curve back by
$\pi_m(z,w)=(z^m,w)$.  The resulting polynomial is
$$
f_m(z,w)=\left(z^m-\frac3{10}\right)w-(z^m+3)(z^m+10).
$$
Its Newton polygon is
$$
\Delta_m=
\operatorname{conv}\{(0,0),(2m,0),(m,1),(0,1)\},
$$
which is again neither a triangle nor a parallelogram.  The curve is
irreducible because its equation is linear in $w$ and its coefficient
$z^m-3/10$ is relatively prime to $(z^m+3)(z^m+10)$ in the Laurent
polynomial ring.  It is smooth: if $(f_m)_w=z^m-3/10=0$, then
$f_m=-(z^m+3)(z^m+10)\ne0$, while at every other point the equation solves
for $w$ as a regular rational function of $z$.

Writing $C_m=V(f_m)$ and $H=V(h)$, one obtains
$$
\mathscr A_{C_m}=
\{(x,y):(mx,y)\in\mathscr A_H\}
$$
and
$$
\mathscr C_{C_m}
=\partial\mathscr A_{C_m}.
$$
Nevertheless, a generic interior logarithmic fiber of $C_m$ has $2m$
points.  Hence $C_m$ is not simple Harnack and is not torus-equivalent to any
simple Harnack curve.

This example also clarifies what the equality
$\mathscr C_f=\partial\mathscr A_f$ does and does not imply.  It says that
the logarithmic map has no critical value strictly inside the amoeba.  It
does not say that a regular interior point has only two inverse images.  A
finite monomial covering creates several unramified sheets above the same
amoeba while introducing no new critical values.  Thus boundary-only
critical values are weaker than the simple Harnack condition.

For a prescribed polygon $\Delta$, the construction applies whenever
$\Delta=A^T\Delta_0$ for a lattice polygon $\Delta_0$ and an integer matrix
$A$ with $|\det A|>1$, and one can choose the Harnack polynomial so that its
monomial pullback is connected.  Such a polygon has a nontrivial lattice
covering structure.  If $\Delta$ is lattice-primitive, meaning that no such
factorization exists, this construction does not apply.  The equality
$\mathscr C_f=\partial\mathscr A_f$ for a non-Harnack curve with that fixed
primitive polygon then becomes a separate classification problem and cannot
be decided from the shape of $\Delta$ alone.

The conclusion is therefore unambiguous: smooth curves not torus-equivalent
to simple Harnack curves can have their complete amoeba contour equal to
their amoeba boundary.  The family $f_m$ above provides explicit examples,
and its generic logarithmic fiber cardinality $2m$ certifies that these
examples lie outside the simple Harnack class even modulo torus
automorphisms.

%%%%%%%%%%%%%%%%%%%%%%%%%%%%%%%%%%%%%%%%%%%%%%%%%%%%%%%%%%%%%%%%%%%%%%%%

\section{A Trapezoidal Example with Contour Equal to the Amoeba Boundary}
 
\begin{theorem}
Let
$f(z,w)=(z+1)w+(z-1)(z-2)=zw+w+z^2-3z+2$.
Then $V_f=\{f=0\}\subset(\C^*)^2$ is smooth and Newton nondegenerate, its Newton polygon is neither a triangle nor a parallelogram, and
$\mathcal C(\mathscr A_f)=\partial\mathscr A_f$.
Moreover, $V_f$ is not obtained by pulling back any plane curve, and hence not by pulling back a smooth simple Harnack curve, through a finite monomial covering of degree greater than one.
\end{theorem}

\begin{proof}
The exponent set of $f$ is
$A_f=\supp(f)=\{(0,0),(1,0),(2,0),(0,1),(1,1)\}$.
Its Newton polygon is therefore
$\Delta=\operatorname{conv}\{(0,0),(2,0),(1,1),(0,1)\}$.
This polygon has four vertices.  Its bottom and top edges are parallel, whereas its other two edges are not parallel, so $\Delta$ is a trapezoid but not a parallelogram.  In particular, it is neither a triangle nor a parallelogram.

The equation can be solved globally for $w$ on the curve:
$w=-\dfrac{(z-1)(z-2)}{z+1}$.
The value $z=-1$ does not occur on $V_f$, because substitution in $f$ gives $f(-1,w)=6$.  Since $f_w=z+1$ is nonzero at every point of $V_f$, the affine curve in the algebraic torus is smooth.

Newton nondegeneracy can be verified face by face.  The bottom-edge polynomial is $(z-1)(z-2)$ and has the two distinct roots $1$ and $2$.  The top-edge polynomial is $w(z+1)$ and its nonmonomial factor has the simple root $-1$.  The left-edge polynomial is $w+2$ and has the simple root $-2$.  The diagonal-edge polynomial is $z^2+zw=z(z+w)$ and its nonmonomial factor has a simple torus root.  Thus every face truncation defines a smooth reduced hypersurface in the corresponding one-dimensional algebraic torus.  Together with the smoothness already proved in $(\C^*)^2$, this shows that $f$ is Newton nondegenerate and that the closure of $V_f$ in the toric surface $X_\Delta$ is a smooth real algebraic curve.

We next determine the logarithmic critical locus exactly.  Use $z$ as a complex parameter and write
$w(z)=-\dfrac{(z-1)(z-2)}{z+1}$.
Away from the zeros and poles of this rational function, the differential of the logarithmic map is described by the two complex logarithmic derivatives $1/z$ and $w'(z)/w(z)$.  Its real rank is smaller than two exactly when their quotient is real.  Hence the criticality condition is
$R(z)=z\dfrac{w'(z)}{w(z)}\in\R\cup\{\infty\}$.
Logarithmic differentiation gives
$R(z)=z\left(\dfrac{1}{z-1}+\dfrac{1}{z-2}-\dfrac{1}{z+1}\right)$.
Using $z/(z-a)=1+a/(z-a)$, this becomes
$R(z)=1+\dfrac{1}{z-1}+\dfrac{2}{z-2}+\dfrac{1}{z+1}$.

Let $z=x+iy$.  Since $\operatorname{Im}(1/(z-a))=-y/((x-a)^2+y^2)$ for real $a$, one obtains
$$
\operatorname{Im}R(x+iy)
=-y\left(
\frac{1}{(x-1)^2+y^2}
+\frac{2}{(x-2)^2+y^2}
+\frac{1}{(x+1)^2+y^2}
\right).
$$
The expression in parentheses is strictly positive wherever it is defined.  It follows that $\operatorname{Im}R(z)=0$ if and only if $\operatorname{Im}z=0$.  The exceptional zeros and poles $z=1$, $z=2$, and $z=-1$ are also real points of the toric compactification.  Consequently the logarithmic Gauss map is totally real: the inverse image of $\mathbb P^1_{\R}$ is precisely the real part of the compactified curve.

The parametrization $z\mapsto(z,w(z))$ is an embedding on the algebraic torus because its first coordinate is $z$.  Newton nondegeneracy shows that this embedding extends without singularity to the smooth toric compactification.  The total-reality characterization of Mikhalkin, in the form recalled by Lang, therefore implies that this compactified curve is a smooth simple Harnack curve \cite{Mikhalkin2000,Lang2024}.  One of the defining geometric consequences is that the logarithmic map is at most two-to-one and that its critical image is exactly the topological boundary of its image.  Hence
$\mathcal C(\mathscr A_f)=\partial\mathscr A_f$.

For completeness, the last conclusion can also be understood locally without appealing only to terminology.  Total reality says that every logarithmic critical point belongs to the real curve.  On every component of the complement of the real critical locus, complex conjugation exchanges the two local sheets of the logarithmic map.  The simple Harnack property orders the real arcs cyclically along the toric boundary and prevents a real critical arc from entering the interior of the amoeba.  Thus the real critical image forms the complete boundary, while the nonreal conjugate sheets map to its interior.  This gives both inclusions between the contour and the amoeba boundary.

It remains to prove that this example cannot arise from a nontrivial monomial covering.  The difference lattice of the support is
$\Lambda_f=\langle a-a':a,a'\in A_f\rangle_{\Z}$.
Because $(1,0)=(1,0)-(0,0)$ and $(0,1)=(0,1)-(0,0)$ occur among the support differences, one has $\Lambda_f=\Z^2$.

Let $M\in\operatorname{Mat}_{2\times2}(\Z)$ with $\det M\ne0$, and define
$\Phi_M(z,w)=(z^{m_{11}}w^{m_{12}},z^{m_{21}}w^{m_{22}})$.
This monomial map is a finite covering of degree $|\det M|$.  If, up to multiplication by a Laurent monomial, $f$ were a pullback $h\circ\Phi_M$, then every difference of two exponents of $f$ would belong to $M^T\Z^2$.  Hence $\Lambda_f\subseteq M^T\Z^2$.  When $|\det M|>1$, the lattice $M^T\Z^2$ has index $|\det M|$ in $\Z^2$ and is a proper sublattice.  This contradicts $\Lambda_f=\Z^2$.

The same conclusion holds if the pullback assertion is stated only for zero sets.  The polynomial $f$ is irreducible in the Laurent polynomial ring: it has degree one in $w$, and the relatively prime polynomials $z+1$ and $(z-1)(z-2)$ are its two coefficients.  Two reduced irreducible hypersurfaces in $(\C^*)^2$ with the same zero set have defining Laurent polynomials that differ by a nonzero constant and a Laurent monomial.  Therefore equality $V_f=\Phi_M^{-1}(V_h)$ would yield the factorization already excluded by the full difference lattice.

Equivalently, the inverse image under $\Phi_M$ is invariant under its kernel, a finite subgroup of order $|\det M|$.  If a torus translation $(z,w)\mapsto(\xi z,\eta w)$ preserved $V_f$, irreducibility would force
$f(\xi z,\eta w)=c z^uw^v f(z,w)$
for a nonzero constant $c$ and integers $u,v$.  Equality of the finite supports forces $u=v=0$.  Comparing the unique $z^2$ term gives $c=\xi^2$, comparing the nonzero constant gives $c=1$, comparing the $z$ term gives $\xi=1$, and comparing the $w$ term gives $\eta=1$.  Thus the translation stabilizer is trivial, whereas a pullback through a covering of degree greater than one would have a nontrivial stabilizer.  This completes the proof.
\end{proof}

\begin{remark}
This example meets the additional polygon requirement and has full exponent lattice.  It is itself a smooth simple Harnack curve, but it is not a pullback through a covering of degree greater than one.  Therefore it answers the present question exactly.  If one additionally requires the curve not to be torus-equivalent to any simple Harnack curve, that is a stronger problem and is not settled by this construction.
\end{remark}

%%%%%%%%%%%%%%%%%%%%%%%%%%%%%%%%%%%%%%%%%%%%%%%%%%%%%%%%%%%%%%%%%%%%%%%%%

\section{Smooth Boundary-Only Contours and Non-Harnack Curves}

Let $\Delta\subset\mathbb R^2$ be a fixed two-dimensional lattice polygon.
It is more precise to write $\Delta\subset\mathbb R^2$ with vertices in
$\mathbb Z^2$, rather than $\Delta\subset\mathbb Z^2$, because a polygon is
a two-dimensional convex subset of the real plane.  Let
$\mathcal S_\Delta$ be the set of Laurent polynomials $f$ in two variables
such that $\Newt(f)=\Delta$.  For $f\in\mathcal S_\Delta$, put
$V_f=\{f=0\}\subset(\mathbb C^*)^2$. The amoeba and its contour are denoted as before
$\mathscr A_f=\Logmap(V_f)$ and
$\mathcal C(\mathscr A_f)=\Logmap(\Crit(\Logmap|_{V_f}))$.

There are two meanings of the assertion that $V_f$ is not Harnack.  In the
literal meaning, $V_f$ is not a simple Harnack curve for the standard real
structure $(z,w)\mapsto(\overline z,\overline w)$.  In the stronger meaning,
$V_f$ is not equivalent to any simple Harnack curve under a complex torus
translation or a monomial torus automorphism.  The answer to the question is
affirmative for every $\Delta$ in the first meaning.  In the stronger
meaning, it is affirmative for many polygons, but it is not a consequence of
the polygon alone and requires a separate construction.

\begin{theorem}
For every two-dimensional lattice polygon $\Delta$,   except unimodular triangles, there exists a Laurent
polynomial $f\in\mathcal S_\Delta$ such that $V_f$ is smooth and Newton
nondegenerate with respect to $\Delta$, its logarithmic critical locus is
smooth, its contour is a smooth embedded real-analytic curve, and
$$
\mathcal C(\mathscr A_f)=\partial\mathscr A_f.
$$
If the polygon is not unimodular triangle, the polynomial $f$ can be chosen nonreal up to multiplication by a nonzero
constant.  Therefore, $V_f$ is not a Harnack curve with respect to the
standard real structure.
\end{theorem}

\begin{proof}
For every two-dimensional lattice polygon $\Delta$, there exists a smooth
simple Harnack curve in the toric surface $X_\Delta$.  It can be defined in
the algebraic torus by a real Laurent polynomial
$h(z,w)=\sum_{(j,k)\in A}a_{jk}z^jw^k$, where
$A\subset\Delta\cap\mathbb Z^2$, all $a_{jk}$ are real,
$\operatorname{conv}(A)=\Delta$, and the toric compactification is smooth,
Newton nondegenerate, and transverse to the toric boundary.  Such curves can
be obtained by simple Harnack patchworking.

For a smooth simple Harnack curve, the logarithmic Gauss map
$\gamma_h=[zh_z:wh_w]$ is totally real.  Therefore its inverse image of
$\mathbb P^1_{\mathbb R}$ is precisely the real part of the curve.  By the
logarithmic critical-point criterion, this inverse image is the complete
critical locus of the logarithmic map.  The logarithmic map realizes the
quotient of the complex curve by complex conjugation and is at most
two-to-one.  Its real locus maps injectively onto the amoeba boundary.  Hence
$$
\mathcal C(\mathscr A_h)=\partial\mathscr A_h.
$$
Choosing the simple Harnack curve smooth and transverse to the toric boundary
ensures that this common set is a smooth embedded real-analytic curve in the
finite logarithmic plane.

Choose real numbers $\alpha$ and $\beta$ and consider the unitary torus
translation
$T_{\alpha,\beta}(z,w)=(e^{i\alpha}z,e^{i\beta}w)$.  Define
$$
f_{\alpha,\beta}(z,w)
=h(e^{-i\alpha}z,e^{-i\beta}w)
=\sum_{(j,k)\in A}
a_{jk}e^{-i(j\alpha+k\beta)}z^jw^k.
$$
No coefficient vanishes and no exponent changes, so
$\Newt(f_{\alpha,\beta})=\Newt(h)=\Delta$.  Moreover,
$V_{f_{\alpha,\beta}}=T_{\alpha,\beta}(V_h)$.  The map
$T_{\alpha,\beta}$ is an algebraic automorphism of the complex torus and
extends compatibly to the toric compactification.  It consequently preserves
smoothness, Newton nondegeneracy, and transversality to the toric boundary.

Since $|e^{i\alpha}z|=|z|$ and $|e^{i\beta}w|=|w|$, one has the exact
identity $\Logmap\circ T_{\alpha,\beta}=\Logmap$.  It follows immediately
that
$$
\mathscr A_{f_{\alpha,\beta}}=\mathscr A_h.
$$
Differentiating the same identity and using the invertibility of
$dT_{\alpha,\beta}$ shows that a point $p\in V_h$ is logarithmically critical
if and only if $T_{\alpha,\beta}(p)$ is logarithmically critical on
$V_{f_{\alpha,\beta}}$.  Thus
$S_{f_{\alpha,\beta}}=T_{\alpha,
\beta}(S_h)$ and
$$
\mathcal C(\mathscr A_{f_{\alpha,\beta}})
=\mathcal C(\mathscr A_h)
=\partial\mathscr A_h
=\partial\mathscr A_{f_{\alpha,\beta}}.
$$
The contour and the critical locus remain smooth because they are transported
by diffeomorphisms, and the contour itself is unchanged as a subset of
$\mathbb R^2$.

It remains to choose the phases so that the polynomial is not real up to a
common scalar.  If a number $\lambda\in\mathbb C^*$ made every coefficient
of $\lambda f_{\alpha,\beta}$ real, then for every two exponents
$(j,k),(j',k')\in A$ one would have
$$
(j-j')\alpha+(k-k')\beta\in\pi\mathbb Z.
$$
Since $A$ affinely spans a two-dimensional polygon, its exponent
differences span a rank-two subgroup of $\mathbb Z^2$.  The displayed
congruences restrict $(\alpha,\beta)$ to a countable union of proper affine
subsets of $\mathbb R^2$.  Choose $(\alpha,\beta)$ outside that union.  Then
$f_{\alpha,\beta}$ is not real up to a scalar, so its curve is not a simple
Harnack curve for the standard real structure.  All the other asserted
properties have already been proved.
\end{proof}

The theorem is not only a formal change of
coefficients: the resulting defining polynomial is genuinely nonreal in
those coordinates.  Nevertheless, the curve is a unitary translate of a
simple Harnack curve.  This qualification matters if Harnack curves are being
classified modulo torus transformations.

\begin{proposition}
There exist smooth curves which are not torus-equivalent to any simple
Harnack curve and nevertheless have a smooth contour equal to the amoeba
boundary.  One explicit family is
$$
f_m(z,w)=
\left(z^m-\frac3{10}\right)w-(z^m+3)(z^m+10),
\qquad m>1.
$$
Its Newton polygon is
$\Delta_m=\operatorname{conv}\{(0,0),(2m,0),(m,1),(0,1)\}$.
\end{proposition}

\begin{proof}
Consider first
$$
h(u,w)=\left(u-\frac3{10}\right)w-(u+3)(u+10).
$$
The curve $H=V_h$ is parametrized by
$u=t$ and $w=(t+3)(t+10)/(t-3/10)$.  Put $t=re^{i\theta}$.  Direct
differentiation gives
$$
\frac{\partial}{\partial\theta}
\log\left|\frac{(t+3)(t+10)}{t-3/10}\right|
=r\sin\theta\left(
-\frac3{|t+3|^2}-\frac{10}{|t+10|^2}
-\frac{3/10}{|t-3/10|^2}
\right).
$$
The expression in parentheses is strictly negative.  Therefore, on every
regular circle $|t|=r$, the second logarithmic coordinate has exactly one
maximum and one minimum.  Every horizontal amoeba slice is the interval
between these two critical values.  Consequently, the complete contour of
$H$ is smooth and equals $\partial\mathscr A_h$.  This is also the standard
amoeba-map characterization of a rational simple Harnack curve.

Let $\pi_m(z,w)=(z^m,w)$ and observe that
$V_{f_m}=\pi_m^{-1}(H)$.  The monomial map is an unramified covering of the
torus of degree $m$, and
$$
\Logmap\circ\pi_m(x,y)=(m x,y)
$$
in logarithmic coordinates.  Hence the amoeba and contour of $V_{f_m}$ are
the inverse images of those of $H$ under the invertible real linear map
$(x,y)\mapsto(mx,y)$.  It follows that
$$
\mathcal C(\mathscr A_{f_m})=\partial\mathscr A_{f_m},
$$
and this common curve is smooth.

The polynomial $f_m$ is linear in $w$, and the two coefficients
$z^m-3/10$ and $(z^m+3)(z^m+10)$ are relatively prime.  Thus the curve is
irreducible.  It is smooth in the torus because, if
$(f_m)_w=z^m-3/10=0$, then $f_m=-(z^m+3)(z^m+10)\ne0$; otherwise the
equation solves regularly for $w$.

At a generic interior point of the amoeba of $H$, the logarithmic fiber has
two points.  Each has $m$ inverse images under $\pi_m$, all with the same
logarithmic image.  Therefore a generic interior logarithmic fiber of
$V_{f_m}$ contains $2m>2$ points.  A simple Harnack curve has logarithmic map
at most two-to-one.  Torus translations and monomial torus automorphisms
preserve generic logarithmic fiber cardinality.  Consequently, $V_{f_m}$ is
not torus-equivalent to any simple Harnack curve.
\end{proof}

The proposition applies only when the prescribed polygon has the indicated
form, or more generally when it is the image of a Harnack Newton polygon
under an integer matrix of determinant greater than one.  For an arbitrary
fixed lattice polygon $\Delta$, it does not follow directly from the
combinatorics of $\Delta$ that there is a curve outside the entire
torus-equivalence class of simple Harnack curves with
$\mathcal C(\mathscr A_f)=\partial\mathscr A_f$.  That stronger problem
requires analyzing the coefficient space for the chosen polygon.

%%%%%%%%%%%%%%%%%%%%%%%%%%%%%%%%%%%%%%%%%%%%%%%%%%%%%%%%%%%%%%%%%%%%%%%%%%

The following example is a Harnack curve.
Let $\Delta=\operatorname{conv}\{(0,0),(4,0),(1,2),(0,2)\}$ and put
$t=3/10$.  The curve used in the balanced figure is the Harnack polynomial
$$
\begin{aligned}
h_t(z,w)={}&1+t z+t^4z^2+t^9z^3+t^{16}z^4\\
&+\left(t-2t^2z+\frac{t^5}{2}z^2\right)w+(t^4+t^5z)w^2.
\end{aligned}
$$
Its signs are the Harnack signs $(-1)^{ij}$ and its coefficient valuations are
$\lambda(i,j)=i^2+j^2$.  The lower convex hull of the lifted lattice points
induces a regular unimodular triangulation of $\Delta$ consisting of ten
triangles.  Thus this is the standard Harnack patchworking construction.
The positive factors $2$ and $1/2$ on the two interior coefficients keep the
curve in the same Harnack chamber while balancing the two bounded
amoeba-complement components corresponding to $(1,1)$ and $(2,1)$.

For an expression with rational coefficients one may write
$$
\begin{aligned}
h_{3/10}(z,w)={}&1+\frac{3z}{10}+\frac{81z^2}{10000}
+\frac{19683z^3}{10^9}+\frac{43046721z^4}{10^{16}}\\
&+\left(\frac3{10}-\frac{9z}{50}+\frac{243z^2}{200000}\right)w
+\left(\frac{81}{10000}+\frac{243z}{100000}\right)w^2.
\end{aligned}
$$
Multiplying by $10^{16}$ gives the integral polynomial
$$
\begin{aligned}
H(z,w)={}&10^{16}+3\cdot10^{15}z+81\cdot10^{12}z^2
+196830000000z^3+43046721z^4\\
&+(3\cdot10^{15}-18\cdot10^{14}z+121500000000z^2)w\\
&+(81\cdot10^{12}+243\cdot10^{11}z)w^2.
\end{aligned}
$$
Every vertex monomial of $\Delta$ occurs and no exponent lies outside
$\Delta$, hence $\operatorname{Newt}(H)=\Delta$.

 \begin{figure}[ht]
\centering
\includegraphics[width=.3\textwidth]{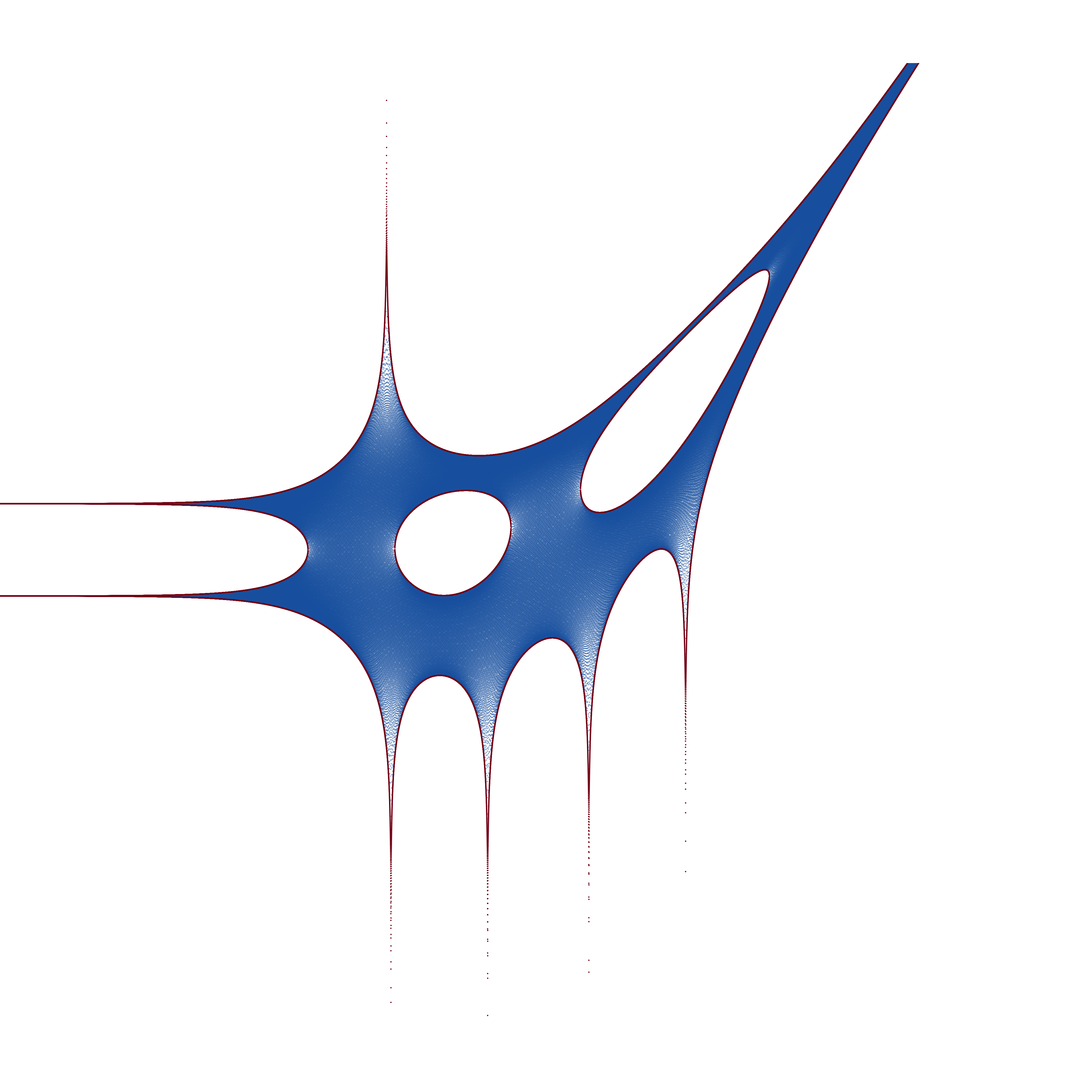}\qquad
\caption{Dark-color amoeba and clear red logarithmic critical values of the Harnack curve defined by the  polynomial $H$.}
\end{figure}

%%%%%%%%%%%%%%%%%%%%%%%%%%%%%%%%%%%%%%%%%%%%%%%%%%%%%%%%%%%%%%%%%%%%%%%%

\section{A Smooth Non-Harnack Curve with Contour Equal to the Amoeba Boundary}
 
 In this section, we study the question when the boundary of the amoeba is equal to its contour?
The following construction answers the question exactly if a smooth curve in $(\C^*)^2$ is allowed to be reducible.  The qualification is essential: the three irreducible components are pairwise disjoint inside the algebraic torus, so their union is smooth there, but their toric compactifications meet on the toric boundary.  Thus this example does not answer the stronger problem in which the curve is required to be irreducible and its toric compactification is required to be smooth and Newton nondegenerate.

Set
$\alpha=\dfrac{3+4i}{5}$ and $\beta=\dfrac{5+12i}{13}$.
Both numbers have absolute value one, they are distinct from one, and they are distinct from each other.  Define
$f(z,w)=(1+z+w)(1+z+\alpha w)(1+z+\beta w)$.
An expanded form, useful for reading the support, is
$$
f(z,w)=(1+z)^3+\frac{129+112i}{65}(1+z)^2w
+\frac{31+168i}{65}(1+z)w^2
+\frac{-33+56i}{65}w^3.
$$

\begin{theorem}
Let $C=V_f\subset(\C^*)^2$.  Then $C$ is a smooth, reducible plane curve, its Newton polygon is $3\Delta_2=\operatorname{conv}\{(0,0),(3,0),(0,3)\}$ and contains the interior lattice point $(1,1)$, and its amoeba satisfies
$\mathcal C(\mathscr A_f)=\partial\mathscr A_f$.
The curve $C$ is not torus-equivalent to a simple Harnack curve and is not the pullback of a smooth simple Harnack curve through any finite monomial covering of degree greater than one.
\end{theorem}

\begin{proof}
For $t\in\{1,\alpha,\beta\}$, let $L_t$ denote the curve $1+z+tw=0$ in $(\C^*)^2$.  Then $C=L_1\cup L_\alpha\cup L_\beta$.  Each $L_t$ is smooth because its gradient is $(1,t)$ and is nowhere zero.

The components are pairwise disjoint in the algebraic torus.  Indeed, if a point belonged to both $L_s$ and $L_t$ for two distinct parameters $s$ and $t$, subtraction of the two equations would give $(s-t)w=0$.  Since $s\ne t$ and $w\ne0$ in $(\C^*)^2$, this is impossible.  At every point of $C$ exactly one factor of $f$ vanishes.  If this factor is $1+z+tw$, differentiation of the product shows that $df$ is the nonzero differential $d(1+z+tw)$ multiplied by the two nonvanishing remaining factors.  Hence $df$ does not vanish at any point of $C$, proving that the entire reducible curve is smooth in $(\C^*)^2$.

We now calculate the Newton polygon.  Since the four coefficients appearing in the displayed expansion are nonzero, expanding the powers of $1+z$ shows that every monomial $z^aw^b$ with $a\geq0$, $b\geq0$, and $a+b\leq3$ occurs with a nonzero coefficient.  Consequently,
$\supp(f)=3\Delta_2\cap\Z^2$ and $\Newt(f)=3\Delta_2$.
The point $(1,1)$ satisfies $1>0$, $1>0$, and $1+1<3$, so it lies in the relative interior of this Newton polygon.  Thus the polygon has an interior lattice point, as required.

Consider first the Harnack line $L_1$, defined by $h(z,w)=1+z+w$.  Its amoeba has the elementary triangle-inequality description
$\mathscr A_h=\{(x,y)\in\R^2:e^x+e^y\geq1,\ 1+e^x\geq e^y,\ 1+e^y\geq e^x\}$.
The boundary is the union of the three loci on which one of these inequalities is an equality.  The logarithmic Gauss map of the line is $[z:w]$.  A point of the line is logarithmically critical precisely when $z/w\in\R$.  Together with $1+z+w=0$, this forces $z$ and $w$ to be real.  Their absolute values then form a degenerate triangle with side lengths $1$, $|z|$, and $|w|$.  Conversely, every degenerate triangle of these three lengths is realized by real nonzero $z$ and $w$ satisfying $1+z+w=0$.  It follows directly that
$\mathcal C(\mathscr A_h)=\partial\mathscr A_h$.

For a parameter $t$ of absolute value one, the torus automorphism $\rho_t(z,w)=(z,tw)$ preserves the logarithmic map, because $\log|tw|=\log|w|$.  It carries $L_t$ onto $L_1$.  Therefore
$\mathscr A_{L_t}=\mathscr A_h$ and $\mathcal C(\mathscr A_{L_t})=\mathcal C(\mathscr A_h)$.
Since the zero set of a product is the union of the zero sets of its factors and since the three components are disjoint, the amoeba and contour of $C$ satisfy
$\mathscr A_f=\mathscr A_{L_1}\cup\mathscr A_{L_\alpha}\cup\mathscr A_{L_\beta}=\mathscr A_h$
and
$\mathcal C(\mathscr A_f)=\mathcal C(\mathscr A_{L_1})\cup\mathcal C(\mathscr A_{L_\alpha})\cup\mathcal C(\mathscr A_{L_\beta})=\mathcal C(\mathscr A_h)$.
Combining these identities gives the desired equality
$\mathcal C(\mathscr A_f)=\partial\mathscr A_f$.

The curve is not torus-equivalent to a simple Harnack curve.  Under the standard definition, a smooth simple Harnack curve is an irreducible real algebraic curve embedded in a toric surface in maximal position.  In particular, its complex curve is connected and irreducible.  The curve $C$ has exactly three irreducible components.  A torus automorphism is an algebraic isomorphism and therefore preserves reducibility and the number of irreducible components.  No torus automorphism can consequently transform $C$ into a simple Harnack curve.

It remains to exclude every finite monomial pullback of degree greater than one.  Define the difference lattice of the support by
$\Lambda_f=\langle a-a':a,a'\in\supp(f)\rangle_\Z\subseteq\Z^2$.
Because the support contains $(0,0)$, $(1,0)$, and $(0,1)$, it contains the two differences $(1,0)$ and $(0,1)$.  Hence $\Lambda_f=\Z^2$.

Let $M=(m_{jk})\in\operatorname{Mat}_{2\times2}(\Z)$ have nonzero determinant.  The associated monomial map is defined by

$\Phi_M(z,w)=(z^{m_{11}}w^{m_{12}},z^{m_{21}}w^{m_{22}})$.

The degree of this finite monomial covering is $|\det M|$.  If $g(x,y)=\sum_b c_bx^{b_1}y^{b_2}$, then the exponent of the pullback of $x^{b_1}y^{b_2}$ is $M^Tb$.  Therefore every difference of two exponents of $g\circ\Phi_M$ lies in $M^T\Z^2$.  Multiplication by a Laurent monomial translates the support and does not alter its difference lattice.

Suppose that $C$ were the complete pullback of a smooth simple Harnack curve $V_g$ through $\Phi_M$.  Since $V_g$ is smooth, its defining polynomial is reduced.  The monomial covering is unramified on $(\C^*)^2$, so $g\circ\Phi_M$ is also reduced.  The reduced defining equations of the same hypersurface differ by a nonzero scalar and a Laurent monomial.  It would follow that
$\Lambda_f\subseteq M^T\Z^2$.
If $|\det M|>1$, the sublattice $M^T\Z^2$ has index $|\det M|$ in $\Z^2$ and is proper.  This contradicts $\Lambda_f=\Z^2$.  Hence no finite monomial covering of degree greater than one can produce $C$ as the complete pullback of a smooth simple Harnack curve.

This also shows that the obstruction is unchanged by any preliminary torus equivalence.  A monomial torus automorphism acts on exponent differences by a matrix in $\operatorname{GL}_2(\Z)$ and preserves the lattice index.  Rescaling $z$ and $w$ changes coefficients but not exponent positions.  The full-lattice obstruction therefore survives every allowed torus coordinate change.
\end{proof}

\begin{remark}
The distinction between smoothness in the algebraic torus and smoothness of the toric compactification must not be suppressed.  The three affine components are disjoint because their intersections have been moved to the toric boundary.  Their closures consequently meet there, and the product polynomial is not Newton nondegenerate.  The example is therefore a complete answer under the literal hypotheses stated in the question, but not under the stronger assumptions of irreducibility and a smooth Newton-nondegenerate toric compactification.
\end{remark}
 
%%%%%%%%%%%%%%%%%%%%%%%%%%%%%%%%%%%%%%%%%%%%%%%%%%%%%%%%%%%%%%%%%%%%%%%

\section{Smoothness and Boundary Equality for the Amoeba Contour}
 
Let $ f(z,w)=zw^2+w^2-z^4+10z^3-35z^2+50z-24$ and let $C=V(f)\subset(\mathbb C^*)^2$.  We prove that the complete contour
of the amoeba of $C$ is a smooth embedded real curve and that it is exactly the topological boundary of the amoeba.
Write
$$
P(z)=(z-1)(z-2)(z-3)(z-4).
$$
Then $f(z,w)=(z+1)w^2-P(z)$, so the equation of $C$ is
$$
w^2=R(z),
\qquad
R(z)=\frac{P(z)}{z+1}.
$$
The curve is smooth in the algebraic torus.  Indeed,
$f_w(z,w)=2(z+1)w$.  At every point of $(\mathbb C^*)^2$ one has $w\ne0$.
If $z=-1$, then $f(-1,w)=-P(-1)=-120\ne0$, so no point of $C$ has
$z=-1$.  Hence $f_w$ is nowhere zero on $C$.  In particular, $z$ is a
holomorphic local coordinate everywhere on $C$.

Put $z=re^{i\theta}$, where $r>0$.  Since the two square roots of $R(z)$
differ by sign, they have the same logarithmic image.  Thus the logarithmic
image above the circle $|z|=r$ is described by
$$
x=\log r,
\qquad
y_r(\theta)=\frac12\log|R(re^{i\theta})|.
$$
At radii for which the circle contains a zero or the pole of $R$, this
formula is interpreted on the punctured circle and by taking limits at its
ends.

We first determine the complete logarithmic critical locus.  In the local
real coordinates $(\log r,\theta)$ on the $z$-plane, the logarithmic map is
$(\log r,y_r(\theta))$.  Its Jacobian matrix has first row $(1,0)$, so its
determinant is $\partial y_r/\partial\theta$.  Therefore a point of $C$ is
critical precisely when
$
\dfrac{\partial y_r}{\partial\theta}=0.
$
For a real number $a$, direct differentiation gives
$$
\frac{\partial}{\partial\theta}
\log|re^{i\theta}-a|
=\frac{ar\sin\theta}
{r^2-2ar\cos\theta+a^2}.
$$
Applying this formula to the four zeros $1,2,3,4$ of $P$ and to the pole
$-1$ of $R$ gives
$$
\frac{\partial y_r}{\partial\theta}
=\frac{r\sin\theta}{2}B(r,\theta),
$$
where
$$
B(r,\theta)=
\sum_{a=1}^4\frac{a}{r^2-2ar\cos\theta+a^2}
+\frac{1}{r^2+2r\cos\theta+1}.
$$
Every denominator is the squared modulus of a nonzero complex number at a
point of $C$, and every numerator is positive.  Hence
$B(r,\theta)>0$.  It follows that
$\di
\frac{\partial y_r}{\partial\theta}=0
\, \Longleftrightarrow\, 
\sin\theta=0.
$
Thus every logarithmic critical point lies over the positive or negative
real $z$-axis, and every point of $C$ over these two rays is critical.  This
also proves that there are no additional nonreal angular critical branches.

The sign formula shows more.  On $0<\theta<\pi$ one has
$\partial y_r/\partial\theta>0$, while on
$\pi<\theta<2\pi$ one has
$\partial y_r/\partial\theta<0$.  Consequently, on each regular circle
$|z|=r$, the function $y_r$ increases strictly from its value at
$\theta=0$ to its value at $\theta=\pi$ and then decreases strictly back to
its initial value.  Its image is therefore the closed interval with these
two endpoint values.

Define
$$
Y_+(r)=\frac12\log\left|
\frac{(r-1)(r-2)(r-3)(r-4)}{r+1}\right|
$$
whenever the expression is finite, and
$$
Y_-(r)=\frac12\log\left|
\frac{(r+1)(r+2)(r+3)(r+4)}{1-r}\right|
$$
for $r\ne1$.  These are precisely $y_r(0)$ and $y_r(\pi)$.  Hence, at every
regular radius,
$$
\mathscr A_f\cap\bigl(\{\log r\}\times\mathbb R\bigr)
=\{\log r\}\times
[\min\{Y_+(r),Y_-(r)\},\max\{Y_+(r),Y_-(r)\}].
$$
At $r=1$, the positive endpoint tends to $-\infty$ because $z=1$ is a zero
of $R$, whereas the negative endpoint tends to $+\infty$ because $z=-1$ is
a pole.  The punctured circle consequently maps onto the entire vertical
line $\{0\}\times\mathbb R$.  At $r=2,3,4$, one endpoint tends to
$-\infty$, so the same slice description holds with an unbounded interval.
These limiting cases create no additional finite critical values.

It follows from the slice formula that a finite point of the amoeba boundary
must be one of the two endpoint values $Y_+(r)$ or $Y_-(r)$.  Conversely,
each finite endpoint is a boundary point.  Indeed, strict monotonicity shows
that values immediately between the two endpoints belong to the amoeba,
whereas values immediately beyond an endpoint do not belong to that
vertical slice.  Continuity with respect to $r$ then gives a local one-sided
neighborhood in the complement.  Since the critical values are exactly
these finite endpoint values, we obtain
$$
\mathcal C(\mathscr{A}_f)
=\Logmap\bigl(\Crit(\Logmap|_C)\bigr)
=\partial\mathscr A_f.
$$

It remains to prove that this common set is smooth and embedded.  Each of
the two critical-value families is a graph over $x=\log r$:
$$
\Gamma_+=\{(\log r,Y_+(r)):r>0,\ r\notin\{1,2,3,4\}\}
$$
and
$$
\Gamma_-=\{(\log r,Y_-(r)):r>0,\ r\ne1\}.
$$
On every indicated open interval, $Y_+$ and $Y_-$ are real analytic.  The
parametrizations have first coordinate derivative $1/r\ne0$, so each is a
regular real-analytic immersion.  The excluded radii correspond only to
ends escaping to infinity in $\mathbb R^2$; they are not missing finite
points at which a cusp could occur.

We must also exclude intersections between the two graph families.  If
$Y_+(r)=Y_-(r)$ for some admissible $r>0$, then
$$
\left|\frac{P(r)}{r+1}\right|
=\left|\frac{P(-r)}{1-r}\right|.
$$
All quantities in this equality are real.  Hence either
$P(r)(1-r)=P(-r)(r+1)$ or
$P(r)(1-r)=-P(-r)(r+1)$.  Direct expansion gives, respectively,
$$
P(r)(1-r)-P(-r)(r+1)
=-2r(r^4+45r^2+74)
$$
and
$$
P(r)(1-r)+P(-r)(r+1)
=2(11r^4+85r^2+24).
$$
The first expression is strictly negative for $r>0$, and the second is
strictly positive.  Neither equation has a positive solution.  Therefore
$\Gamma_+$ and $\Gamma_-$ are disjoint in the finite plane.

There are also no self-intersections within either family, because the first
coordinate $\log r$ uniquely determines $r$.  Thus every connected piece is
an embedded real-analytic graph, distinct pieces can meet only at infinity,
and the two families never meet each other.  The contour is consequently a
smooth embedded one-dimensional real-analytic submanifold of $\mathbb R^2$.
We have proved both asserted conclusions: the full logarithmic critical
value set is smooth, and
$
{\mathscr C_f=\partial\mathscr A_f.}
$
\begin{figure}[ht]
\centering
\includegraphics[width=0.22\textwidth]{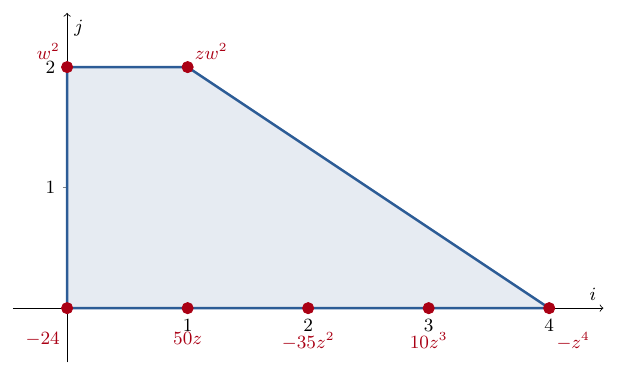}
\qquad
\includegraphics[width=0.32\textwidth]{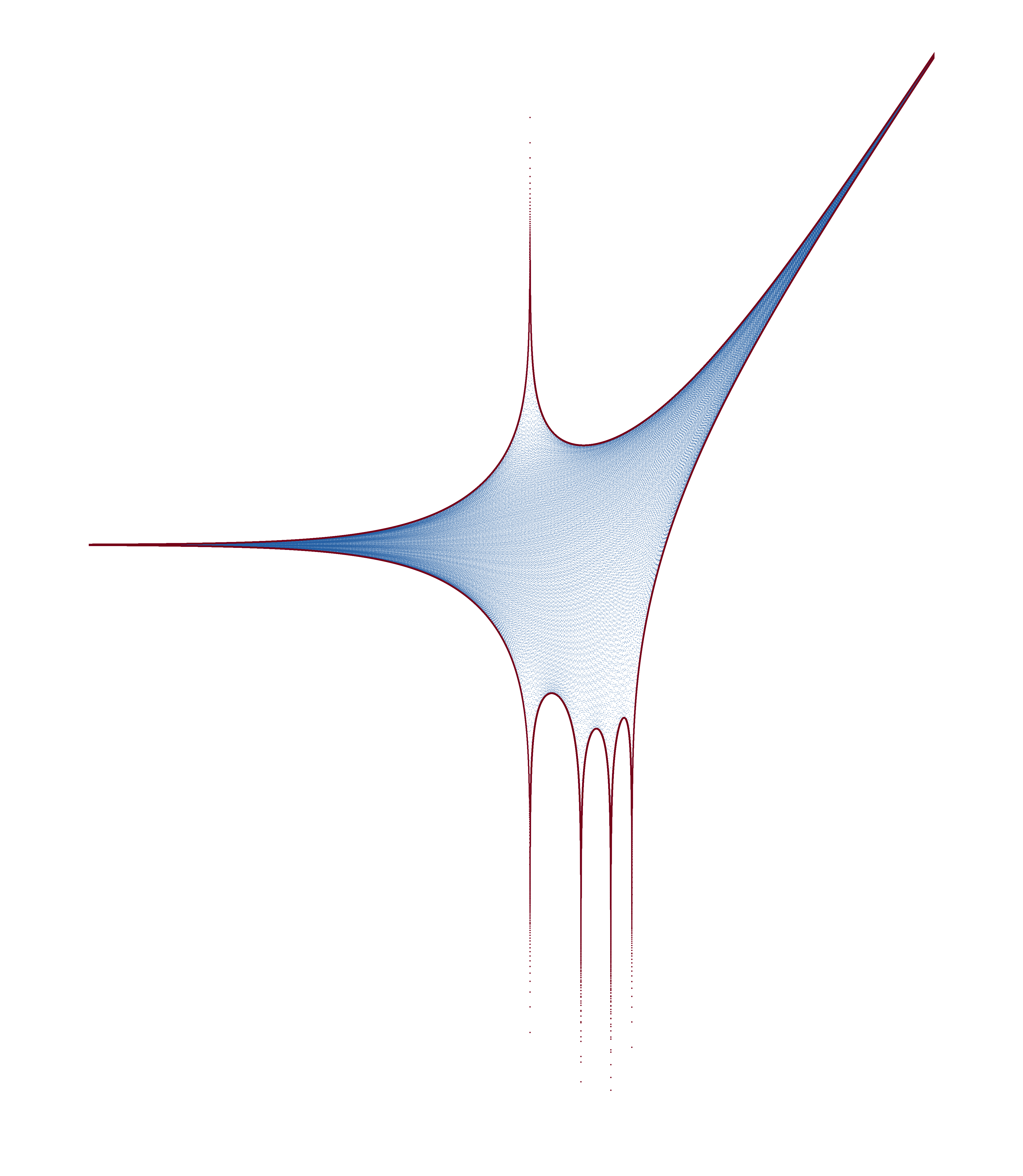}
\caption{Dark-blue amoeba and clear red logarithmic critical values of the curve defined by $f(z,w)=zw^2+w^2-z^4+10z^3-35z^2+50z-24$.}
\end{figure}

%%%%%%%%%%%%%%%%%%%%%%%%%%%%%%%%%%%%%%%%%%%%%%%%%%%%%%%%%%%%%%%%%%%%%%%%%

The Newton polygon is $\Delta=\operatorname{conv}\{(0,0),(4,0),(1,2)\}$.  The plotted polynomial is
$$
f(z,w)=1+\frac{3}{10}z+\left(\frac{3}{10}\right)^4z^2
+\left(\frac{3}{10}\right)^9z^3+\left(\frac{3}{10}\right)^{16}z^4
-\frac{117}{1250}zw+\frac{6561}{10000000}z^2w
+\left(\frac{3}{10}\right)^5zw^2.
$$

\begin{figure}[ht]
\centering
\includegraphics[width=0.42\textwidth]{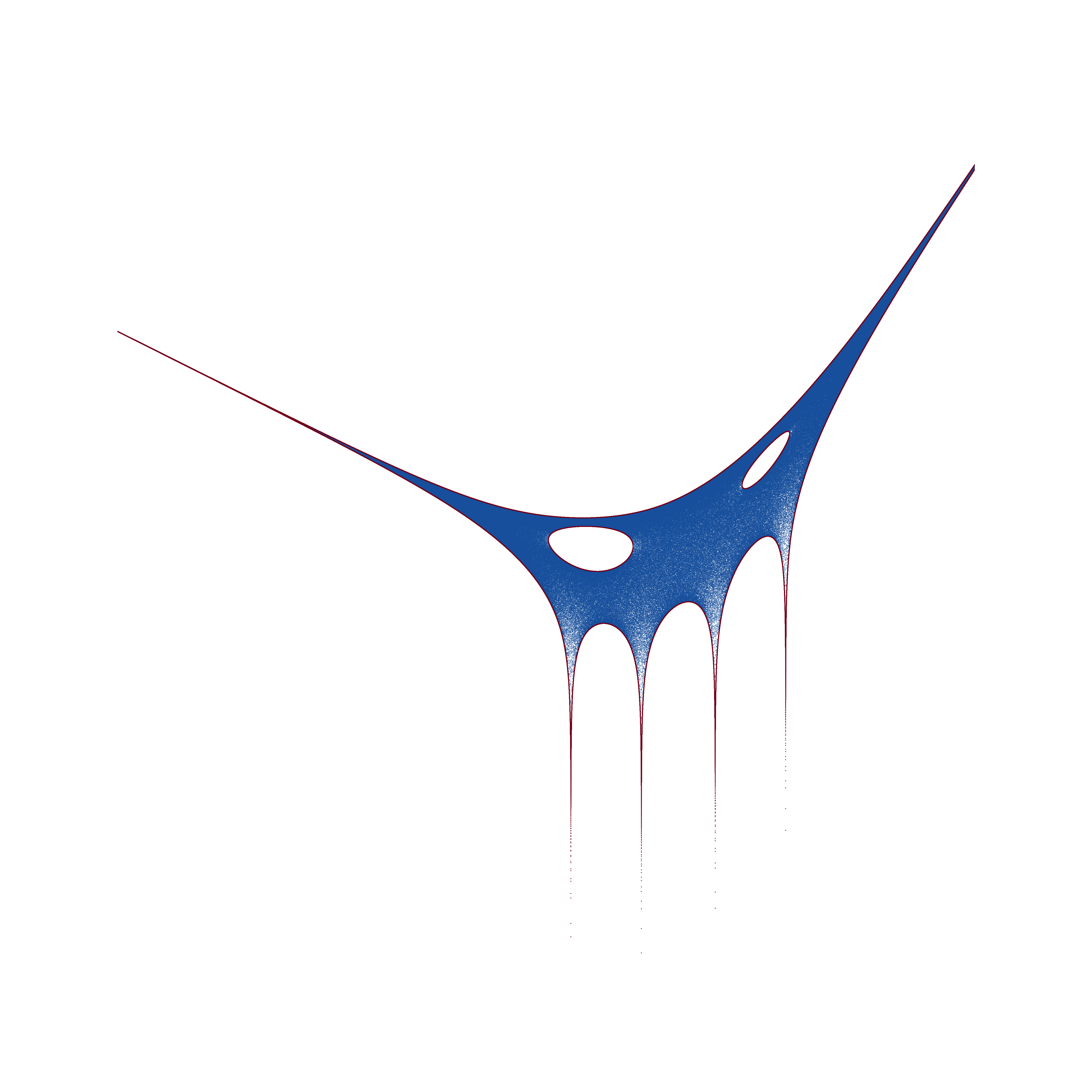}
\caption{Dark-blue amoeba of Harnack curve with Newton polygon $\Delta = \operatorname{Convexhull}\{(0,0), (4,0), (1,2), (0,2) \}$ and defining polynomial $f$.}
\end{figure} 
%%%%%%%%%%%%%%%%%%%%%%%%%%%%%%%%%%%%%%%%%%%%%%%%%%%%%%%%%%%%%%%%%%%%%%%%%%%
%%%%%%%%%%%%%%%%%%%%%%%%%%%%%%%%%%%%%%%%%%%%%%%%%%%%%%%%%%%%%%%%%%%%%%%%%%

\section{Exact Sturm computation}

Let $H$ be a nonzero square-free polynomial with real coefficients.  Its Sturm sequence is the finite sequence of polynomials $S_0,S_1,\ldots,S_m$ defined by $S_0=H$, $S_1=H'$, and
$
S_{j+1}=-\operatorname{rem}(S_{j-1},S_j)
$
for every subsequent index, where $\operatorname{rem}(A,B)$ denotes the remainder in the Euclidean division of $A$ by $B$.  The construction ends when the last nonzero remainder is reached.  If $H$ is square-free, then the final polynomial is a nonzero constant because $\gcd(H,H')=1$, see  for example \cite{BasuPollackRoy2006}.

For a real number $x$ that is not a zero of any relevant polynomial, let $V(x)$ be the number of sign changes in the ordered list
$
S_0(x),S_1(x),\ldots,S_m(x),
$
after all zero entries have been omitted.  Sturm's theorem states that, if $a<b$ and neither endpoint is a zero of $H$, then the number of distinct real roots of $H$ in $(a,b)$ is exactly $V(a)-V(b)$.  Thus the theorem converts the problem of counting roots into the exact determination of finitely many signs.

For the positive half-line, one uses the endpoint signs at $0^+$ and $+\infty$.  Because $H(0)\neq0$ in the present calculation, the notation $0^+$ simply means that the signs are evaluated immediately to the right of zero; they agree with the signs of the nonzero constant terms of the Sturm polynomials.  At $+\infty$, the sign of each $S_j(R)$ is the sign of its leading coefficient.  Consequently, both endpoint sign lists are obtained exactly from integer or rational coefficients.

In the polynomial occurring in the paper, the square-free factor $H$ has degree sixteen.  Its exact Sturm sequence has degrees
$
16,15,14,13,12,11,10,9,8,7,6,5,4,3,2,1,0.
$
The signs of the sequence at $0^+$ are
$
(+,-,+,+,-,-,+,+,+,-,-,+,+,-,-,+,+),
$
and the signs at $+\infty$ are
$
(+,+,+,-,+,+,+,-,-,+,-,-,+,-,-,-,+).
$
The first sign list has eight variations, and the second sign list also has eight variations.  Sturm's theorem therefore gives
$
N_H(0,+\infty)=V(0^+)-V(+\infty)=8-8=0.
$
Hence $H$ has no positive real root.

The adjective \emph{exact} is important.  Every coefficient of $H$ is an integer, every polynomial remainder can be formed over $\mathbb Q$, and every endpoint sign is determined symbolically.  There is therefore no rounding error.  A numerical routine could return approximate complex roots and suggest that none of them is positive and real, but such evidence would depend on tolerances and could fail to detect a root very close to the real axis or a pair of nearly multiple roots.  The Sturm calculation instead proves the root count as an integer identity.

In the application to the logarithmic critical locus, exact elimination gives
$
\operatorname{Res}_z(G,G_z)=\kappa R^7H(R),
$
where $\kappa\neq0$ and $R=r^2$.  A positive value $R>0$ at which the angular critical equation has a multiple root must make this resultant vanish.  The factor $R^7$ vanishes only at $R=0$, while the Sturm calculation proves that $H(R)\neq0$ for every $R>0$.  Therefore the resultant is nonzero on the entire positive radial axis.  It follows that no positive radius supports a ramified angular critical point.

This conclusion does not by itself count the angular critical points; it proves that their number cannot change through the creation, collision, or disappearance of multiple angular roots as $R$ varies over $(0,+\infty)$.  An additional exact Sturm count at one sample radius, chosen as $R=1$, shows that there are exactly two angular critical points there.  Since no positive ramification radius exists, the count remains equal to two for every positive nonexceptional radius.  The zero and pole radii are then checked separately in reciprocal toric coordinates.

%%%%%%%%%%%%%%%%%%%%%%%%%%%%%%%%%%%%%%%%%%%%%%%%%%%%%%%%%%%%%%%%%%%%%%%%%%

\subsection{A short exact Sturm computation}

Consider the polynomial $H(x)=x^3-x-1$.  We shall use Sturm's theorem to prove exactly that $H$ has one positive real root, that this root belongs to $(1,2)$, and that $H$ has no negative real root.

The Sturm sequence begins with $S_0(x)=H(x)=x^3-x-1$ and $S_1(x)=H'(x)=3x^2-1$.  Euclidean division gives
$
S_0(x)=\frac{x}{3}S_1(x)-\left(\frac{2x}{3}+1\right).
$
The next Sturm polynomial is the negative of the remainder.  Multiplying it by the positive constant $3$ does not change any sign variation, so we may take $S_2(x)=2x+3$.

Dividing $S_1$ by $S_2$ gives
$
S_1(x)=\left(\frac{3x}{2}-\frac94\right)S_2(x)+\frac{23}{4}.
$
The negative of this remainder is $-23/4$.  Again multiplying by a positive constant does not change signs, so we take $S_3(x)=-23$.  The complete Sturm sequence is therefore
$
S_0=x^3-x-1,\qquad S_1=3x^2-1,\qquad S_2=2x+3,\qquad S_3=-23.
$

Let $V(a)$ denote the number of sign changes in the list $S_0(a),S_1(a),S_2(a),S_3(a)$, after any zero entry is omitted.  At $0$ the signs are
$
(-,-,+,-),
$
so $V(0)=2$.  At $+\infty$, the signs are determined by the leading terms and are
$
(+,+,+,-),
$
so $V(+\infty)=1$.  Sturm's theorem now gives
$
N_H(0,+\infty)=V(0)-V(+\infty)=2-1=1.
$
Hence $H$ has exactly one positive real root.

The same computation isolates this root.  At $x=1$, one has $S_0(1)=-1$, $S_1(1)=2$, $S_2(1)=5$, and $S_3(1)=-23$.  The sign list is $(-,+,+,-)$ and therefore $V(1)=2$.  At $x=2$, one has $S_0(2)=5$, $S_1(2)=11$, $S_2(2)=7$, and $S_3(2)=-23$.  The sign list is $(+,+,+,-)$ and therefore $V(2)=1$.  Consequently,
$
N_H(1,2)=V(1)-V(2)=2-1=1.
$
Thus the unique positive root lies in $(1,2)$.

For completeness, at $-\infty$ the signs are
$
(-,+,-,-),
$
which have two variations.  Since $V(0)=2$, Sturm's theorem gives
$
N_H(-\infty,0)=V(-\infty)-V(0)=2-2=0.
$
Therefore $H$ has no negative real root.  Moreover, $H(0)=-1\neq0$, so zero is not a root.  We conclude that $H(x)=x^3-x-1$ has exactly one real root, this root is positive, and it belongs to the interval $(1,2)$.

Every operation above is performed with integers and rational numbers.  This is why the computation is exact: no decimal approximation of the root and no numerical tolerance are used.

%%%%%%%%%%%%%%%%%%%%%%%%%%%%%%%%%%%%%%%%%%%%%%%%%%%%%%%%%%%%%%%%%%%%%%%%%%

%%%%%%%%%%%%%%%%%%%%%%%%%%%%%%%%%%%%%%%%%%%%%%%%%%%%%%%%%%%%%%%%%%%%%%%%%%%

\section{Exponent Lattices and Finite Monomial Coverings}
%\section{Appendix: Conjecture}
 
Let
$f(z,w)=\sum_{a\in A}c_a z^{a_1}w^{a_2}$, where $c_a\neq0$ and
$A=\supp(f)\subset\mathbb Z^2$.  
The exponent-difference lattice of $f$ is
$$
L_f=\left\langle a-b:\ a,b\in\supp(f)\right\rangle_{\mathbb Z}
\subseteq\mathbb Z^2.
$$
The brackets with the subscript $\mathbb Z$ mean that arbitrary integer
linear combinations are allowed.  Thus $L_f$ consists of all vectors of the
form
$$
n_1(a_1-b_1)+\cdots+n_r(a_r-b_r),
\qquad n_1,\ldots,n_r\in\mathbb Z,
$$
where the $a_j$ and $b_j$ belong to the support of $f$.

\begin{definition}
 We say that the exponent differences generate $\mathbb Z^2$ 
 when $L_f=\mathbb Z^2$.  Equivalently, every lattice vector
$(m,n)\in\mathbb Z^2$ can be expressed as an integer linear combination of
differences of exponents occurring in $f$. 
\end{definition}

One may fix any exponent $a_0\in A$ and use only the differences $a-a_0$.
Indeed,
$a-b=(a-a_0)-(b-a_0)$, so
$
L_f=\left\langle a-a_0:\ a\in A\right\rangle_{\mathbb Z}.
$
Consequently the definition is independent of the chosen base exponent.
Multiplying $f$ by a Laurent monomial translates every exponent by the same
vector and therefore does not change $L_f$.

Consider $f(z,w)=1+z+w+z^2w$.  Its support is
$A=\{(0,0),(1,0),(0,1),(2,1)\}$.  Taking $a_0=(0,0)$ gives the differences
$(1,0)$, $(0,1)$ and $(2,1)$.  Since the first two already generate
$\mathbb Z^2$, one has $L_f=\mathbb Z^2$.

The condition may also hold when neither standard basis vector occurs
directly.  Suppose that the available differences include $(2,1)$ and
$(1,1)$.  Their difference is $(1,0)$, and
$(1,1)-(1,0)=(0,1)$.  Hence these two vectors generate $\mathbb Z^2$, because
$$
\det\begin{pmatrix}2&1\\1&1\end{pmatrix}=1.
$$

In contrast, let $f(z,w)=1+z^2+w^2+z^2w^2$.  Every exponent has two even
coordinates, and every exponent difference belongs to $2\mathbb Z^2$.
In fact $L_f=2\mathbb Z\oplus2\mathbb Z$, which has index $4$ in
$\mathbb Z^2$.  The polynomial factors through the degree-four monomial map
$(z,w)\mapsto(z^2,w^2)$:
$$
f(z,w)=h(z^2,w^2),
\qquad h(u,v)=1+u+v+uv.
$$

As a second example, the lattice generated by $(2,0)$ and $(1,1)$ has index
$2$, because
$$
\left|\det\begin{pmatrix}2&1\\0&1\end{pmatrix}\right|=2.
$$
A polynomial whose translated support lies in this lattice factors through
the degree-two monomial map $(z,w)\mapsto(z^2,zw)$.

There is an effective determinant test.  Assume that the Newton polygon is
two-dimensional.  Choose $a_0\in A$, write
$v_a=a-a_0$, and form all determinants $\det(v_a,v_b)$.  Then
$$
[\mathbb Z^2:L_f]
=\gcd\left\{|\det(v_a,v_b)|:\ a,b\in A\right\}.
$$
Zero determinants may be omitted.  Therefore the exponent differences
generate $\mathbb Z^2$ exactly when the greatest common divisor of these
determinants is $1$.  This is the two-dimensional form of the Smith normal
form criterion.

For instance, consider only the four vertices
$A=\{(0,0),(4,0),(1,2),(0,2)\}$ of the quadrilateral discussed earlier.
Using $(0,0)$ as base point, the nonzero absolute determinants are $8$, $8$
and $2$.  Their greatest common divisor is $2$, so the vertex support alone
generates a sublattice of index $2$, not all of $\mathbb Z^2$.  If the support
also contains $(1,0)$ and $(0,1)$, then the exponent differences contain
$(1,0)$ and $(0,1)$ and the resulting lattice is all of $\mathbb Z^2$.
Thus this property depends on the actual support of the polynomial, not only
on the convex shape of its Newton polygon.

The relation with monomial coverings is exact.  Suppose that $L_f$ has finite
index $q>1$.  Choose a lattice basis $m_1,m_2$ of $L_f$ and let $M$ be the
integer matrix having $m_1$ and $m_2$ as its columns.  Then
$|\det M|=q$.  After choosing $a_0\in\supp(f)$, every exponent can be written
$a=a_0+M\nu_a$ with $\nu_a\in\mathbb Z^2$.  Consequently
$$
f(z,w)=z^{(a_0)_1}w^{(a_0)_2}
h\bigl(z^{m_{11}}w^{m_{21}},z^{m_{12}}w^{m_{22}}\bigr)
$$
for a Laurent polynomial $h$.  Hence, up to an irrelevant Laurent monomial,
the curve is obtained by pulling back $V(h)$ through a finite monomial map of
degree $q$.

Conversely, if $f$ factors through a monomial map whose exponent matrix has
determinant of absolute value $q>1$, then all exponent differences lie in the
proper image lattice of that matrix.  Therefore they cannot generate
$\mathbb Z^2$.

%%%%%%%%%%%%%%%%%%%%%%%%%%%%%%%%%%%%%%%%%%%%%%%%%%%%%%%%%%%%%%%%%%%%%%%%%%%%

\subsection{Exponent Lattices and Finite Monomial Coverings}

Let $f\in\mathbb C[z^{\pm1},w^{\pm1}]$ be a nonzero Laurent polynomial and
write $f(z,w)=\sum_{a\in A}c_a z^{a_1}w^{a_2}$, where
$A=\supp(f)\subset\mathbb Z^2$ and every $c_a$ is nonzero.  We assume that
$\operatorname{Newt}(f)=\operatorname{conv}(A)$ is two-dimensional.  This
assumption implies that the abelian group generated by the differences
$a-b$, with $a,b\in A$, has rank two.

The fourth condition in the proposed statement needs a precise formulation.
The phrase ``does not factor through a nontrivial finite monomial covering
merely because all its exponents lie in a proper affine sublattice'' will mean
that there do not exist $a_0\in\mathbb Z^2$, a matrix
$M\in\operatorname{Mat}_{2\times2}(\mathbb Z)$ with $|\det M|>1$, and a
Laurent polynomial $h\in\mathbb C[u^{\pm1},v^{\pm1}]$ such that
$$
f(z,w)=z^{(a_0)_1}w^{(a_0)_2}
h\bigl(z^{m_{11}}w^{m_{21}},z^{m_{12}}w^{m_{22}}\bigr).
$$
Here the columns of $M$ are $m_1=(m_{11},m_{21})$ and
$m_2=(m_{12},m_{22})$.  The associated monomial map is
$\Phi_M(z,w)=(z^{m_{11}}w^{m_{21}},z^{m_{12}}w^{m_{22}})$.  Since
$\det M\neq0$, this is a finite covering of $(\mathbb C^*)^2$ of degree
$|\det M|$.  With this interpretation, the proposition becomes an exact
lattice-theoretic statement.

\begin{proposition}
Let $f$ be a Laurent polynomial whose Newton polygon is two-dimensional.  Fix
$a_0\in A=\supp(f)$, put $v_a=a-a_0$, and define
$L_f=\langle a-b:a,b\in A\rangle_{\mathbb Z}$.  The following conditions are
equivalent: $L_f=\mathbb Z^2$; the index $[\mathbb Z^2:L_f]$ is one; the
greatest common divisor of the nonzero integers $|\det(v_a,v_b)|$, with
$a,b\in A$, is one; and $f$ has no representation of the preceding form with
$|\det M|>1$.
\end{proposition}

\begin{proof}
We begin by checking that the lattice does not depend on the chosen base
exponent.  For every $a,b\in A$ one has
$a-b=(a-a_0)-(b-a_0)$.  Hence every generator $a-b$ belongs to
$\langle a-a_0:a\in A\rangle_{\mathbb Z}$.  Conversely, since $a_0\in A$,
every vector $a-a_0$ is itself a difference of two elements of $A$.
Therefore
$$
L_f=\langle a-a_0:a\in A\rangle_{\mathbb Z}.
$$
If another base point $a_1\in A$ is used, the same argument gives the same
lattice.  In particular, translating the whole support by one lattice vector,
which corresponds to multiplying $f$ by a Laurent monomial, does not change
$L_f$.

Because $\operatorname{conv}(A)$ is two-dimensional, there exist
$a,b\in A$ for which $v_a$ and $v_b$ are linearly independent over
$\mathbb R$.  It follows that $L_f$ has rank two.  Every rank-two subgroup of
$\mathbb Z^2$ has finite index.  Thus $[\mathbb Z^2:L_f]$ is a positive
integer.

The equivalence between $L_f=\mathbb Z^2$ and
$[\mathbb Z^2:L_f]=1$ follows directly from the definition of index.  Indeed,
the index is the cardinality of the quotient group $\mathbb Z^2/L_f$.  This
quotient has one element exactly when every element of $\mathbb Z^2$ lies in
$L_f$, which is exactly the equality $L_f=\mathbb Z^2$.

We now establish the determinant formula.  Enumerate the support as
$A=\{a_0,a_1,\ldots,a_N\}$ and form the integer matrix
$V=(v_{a_1}\ \cdots\ v_{a_N})\in\operatorname{Mat}_{2\times N}(\mathbb Z)$.
The image of the homomorphism $V:\mathbb Z^N\to\mathbb Z^2$ is precisely
$L_f$.  Since $V$ has rank two, the Smith normal form theorem gives unimodular
matrices $U\in\operatorname{GL}_2(\mathbb Z)$ and
$W\in\operatorname{GL}_N(\mathbb Z)$ and positive integers $d_1,d_2$ with
$d_1\mid d_2$ such that
$$
UVW=\begin{pmatrix}d_1&0&0&\cdots&0\\0&d_2&0&\cdots&0\end{pmatrix}.
$$
Multiplication by $U$ is an automorphism of the target lattice and
multiplication by $W$ is an automorphism of the source lattice.  They do not
change the isomorphism class of the cokernel.  Consequently
$\mathbb Z^2/L_f\cong\mathbb Z/d_1\mathbb Z\oplus
\mathbb Z/d_2\mathbb Z$, and therefore
$[\mathbb Z^2:L_f]=d_1d_2$.

Let $D(V)$ denote the greatest common divisor of all $2\times2$ minors of
$V$, with signs ignored.  At least one such minor is nonzero because $V$ has
rank two.  Elementary integral column operations preserve $D(V)$.  Interchanging
two columns changes the sign of some minors but not their absolute values.
Multiplying one column by $-1$ also changes only signs.  Replacing one column
by itself plus an integer multiple of another column transforms each affected
minor into the old minor plus an integer multiple of another old minor.  This
shows that the new minors generate the same ideal of $\mathbb Z$ as the old
minors.  The same argument applies to unimodular row operations.  Hence
$D(V)=D(UVW)$.

For the Smith normal form matrix, every $2\times2$ minor is zero except for
the minor using the first two nonzero columns, whose absolute value is
$d_1d_2$.  It follows that $D(UVW)=d_1d_2$.  Combining this with the index
calculation gives
$$
[\mathbb Z^2:L_f]
=\gcd\{\,|\det(v_a,v_b)|:a,b\in A\,\}.
$$
Therefore $L_f$ has index one if and only if the indicated greatest common
divisor is one.  Notice that this argument also explains why the formula is
independent of the chosen base point, although that independence already
follows from the intrinsic definition of $L_f$.

It remains to prove the equivalence with the monomial-covering condition.
Assume first that $L_f$ is a proper sublattice of $\mathbb Z^2$.  Since it has
rank two, choose a $\mathbb Z$-basis $m_1,m_2$ of $L_f$ and let $M$ be the
integer matrix with these vectors as columns.  The image $M\mathbb Z^2$ is
$L_f$.  The absolute determinant of a lattice basis equals the index of the
lattice, so $|\det M|=[\mathbb Z^2:L_f]>1$.

For every $a\in A$, the vector $a-a_0$ belongs to $L_f=M\mathbb Z^2$.
Because the columns of $M$ form a basis of $L_f$, there is a unique
$\nu_a=(p_a,q_a)\in\mathbb Z^2$ satisfying $a-a_0=M\nu_a$.  Define
$h(u,v)=\sum_{a\in A}c_a u^{p_a}v^{q_a}$.  This is a Laurent polynomial.
For $u=z^{m_{11}}w^{m_{21}}$ and
$v=z^{m_{12}}w^{m_{22}}$, one has
$$
u^{p_a}v^{q_a}
=z^{m_{11}p_a+m_{12}q_a}w^{m_{21}p_a+m_{22}q_a}
=z^{(a-a_0)_1}w^{(a-a_0)_2}.
$$
Multiplying by $z^{(a_0)_1}w^{(a_0)_2}$ and summing over $a$ gives the exact
identity
$$
f(z,w)=z^{(a_0)_1}w^{(a_0)_2}h\bigl(\Phi_M(z,w)\bigr).
$$
Since $|\det M|>1$, this is a factorization through a nontrivial finite
monomial covering.

For completeness, we verify the degree assertion.  On character lattices,
$\Phi_M$ induces the injective homomorphism
$M:\mathbb Z^2\to\mathbb Z^2$.  Its cokernel has order $|\det M|$.  The
kernel of $\Phi_M$ is the finite group dual to this cokernel, so it also has
$|\det M|$ elements.  A homomorphism of algebraic tori with finite kernel is
finite, and every fiber is a translate of the kernel.  Hence
$\deg(\Phi_M)=|\det M|$.

Conversely, suppose that there exist $a_0$, $M$ and $h$ with
$|\det M|>1$ such that
$f=z^{(a_0)_1}w^{(a_0)_2}h\circ\Phi_M$.  Write
$h(u,v)=\sum_{\nu\in B}d_\nu u^{\nu_1}v^{\nu_2}$ with every $d_\nu\neq0$.
Distinct exponents $\nu$ have distinct images $M\nu$ because $\det M\neq0$.
Thus no two monomials become the same Laurent monomial after substitution,
and no cancellation changes the support.  Consequently
$$
\supp(f)=a_0+M B.
$$
For any two exponents $a_0+M\nu$ and $a_0+M\mu$ in this support, their
difference is $M(\nu-\mu)$.  Therefore $L_f\subseteq M\mathbb Z^2$.  Since
$|\det M|>1$, the image $M\mathbb Z^2$ is a proper sublattice of
$\mathbb Z^2$.  Hence $L_f$ is also proper and cannot equal $\mathbb Z^2$.

We have proved both implications.  A proper exponent-difference lattice gives
a factorization through a finite monomial covering of degree greater than one,
and any such factorization forces the exponent-difference lattice to be
proper.  Combining this conclusion with the index criterion and the
determinant formula proves all the claimed equivalences.
\end{proof}

%%%%%%%%%%%%%%%%%%%%%%%%%%%%%%%%%%%%%%%%%%%%%%%%%%%%%%%%%%%%%%%%%%%%%%%%%%
%%%%%%%%%%%%%%%%%%%%%%%%%%%%%%%%%%%%%%%%%%%%%%%%%%%%%%%%%%%%%%%%%%%%%%%%%%
%%%%%%%%%%%%%%%%%%%%%%%%%%%%%%%%%%%%%%%%%%%%%%%%%%%%%%%%%%%%%%%%%%%%%%%%%%

\subsection{The contour--boundary rigidity question}

Let $f$ be a Laurent polynomial in two variables, let $C=V(f)\subset(\C^*)^2$, and denote by $\mathscr A_f=\Log(C)$ its amoeba.  The contour $\mathcal C(\mathscr A_f)$ is the image under the logarithmic map of the critical locus of $\Log|_C$.  Since the boundary of the amoeba is contained in its contour under the usual regularity assumptions, the equality $\mathcal C(\mathscr A_f)=\partial\mathscr A_f$ means that every logarithmic critical value lies on the boundary and that no critical-value branch is contained in the interior of the amoeba.

\begin{question}[Contour--boundary rigidity]\label{ques:contour-boundary-rigidity}
Let $C=V(f)\subset(\C^*)^2$ be an irreducible smooth curve whose compactification in the toric surface $X_{\Newt(f)}$ is smooth and Newton nondegenerate.  Assume that the logarithmic critical locus of $\Log|_C$ is smooth and that $\mathcal C(\mathscr A_f)=\partial\mathscr A_f$.  If the exponent differences of $\supp(f)$ generate $\Z^2$, must $C$ be torus-equivalent to a smooth simple Harnack curve?

More generally, if the exponent-difference lattice is a proper finite-index sublattice of $\Z^2$, must the primitive reduction of $C$ be torus-equivalent to a smooth simple Harnack curve?  Equivalently, must $C$ be obtained from a smooth simple Harnack curve by a finite monomial covering, up to torus equivalence and multiplication of $f$ by a Laurent monomial?
\end{question}

The most accurate statement of the relationship between Question 1 and the challenge question of Lang, Shapiro, and Shustin in \cite{LangShapiroShustin2021} is therefore the following.  The contour--boundary rigidity question is directly motivated by the final challenge of Lang, Shapiro, and Shustin, because it investigates the zero-residual-contour endpoint  $\mathcal R_f=\mathcal C(\mathscr A_f)\setminus\partial\mathscr A_f$ of their proposed compensation phenomenon.  It is nevertheless a distinct and more specialized problem: their challenge seeks a quantitative comparison among several kinds of topological and singularity complexity, whereas the rigidity question asks for an algebro-geometric classification under exact contour--boundary equality.  A positive answer would provide an extremal rigidity theorem compatible with their picture; a negative answer would show that their desired compensation rule cannot depend only on the set-theoretic difference between the contour and the amoeba boundary.

%%%%%%%%%%%%%%%%%%%%%%%%%%%%%%%%%%%%%%%%%%%%%%%%%%%%%%%%%%%%%%%%%%%%%%%%%%
%%%%%%%%%%%%%%%%%%%%%%%%%%%%%%%%%%%%%%%%%%%%%%%%%%%%%%%%%%%%%%%%%%%%%%%%%%
%%%%%%%%%%%%%%%%%%%%%%%%%%%%%%%%%%%%%%%%%%%%%%%%%%%%%%%%%%%%%%%%%%%%%%%%%%

\section{Theorem: Negative Answer to the Contour--boundary rigidity Question}

The following theorem is a  negative answer  to the contour--boundary rigidity. 
 It shows that the equality between the complete contour and the boundary of an amoeba does not characterize smooth simple Harnack curves, even when one assumes irreducibility, smoothness, Newton nondegeneracy, smoothness of the logarithmic critical locus, smoothness of the contour, and primitive exponent support.

Consider the Laurent polynomial
$f_0(z,w)=1+z+(1/100-i/5000)z^2+w+(1/20-7i/200)zw$.
Put $p=1/100-i/5000$, $q=1/20-7i/200$, $P(z)=1+z+pz^2$ and
$Q(z)=1+qz$.  Thus $f_0=P+Qw$, and the curve $C=V(f_0)$ is parametrized by
$w=-P(z)/Q(z)$.  Its Newton polygon is
$\Delta=\operatorname{conv}\{(0,0),(2,0),(1,1),(0,1)\}$.

\begin{theorem}\label{thm:main}
Let
$\di
f_0(z,w)=1+z+\left(\frac{1}{100}-\frac{i}{5000}\right)z^2+w+
\left(\frac{1}{20}-\frac{7i}{200}\right)zw
$
and let $C=V(f_0)\subset(\C^*)^2$.  Its Newton polygon is $\Delta=\operatorname{conv}\{(0,0),(2,0),(1,1),(0,1)\}$.  Then $C$ is irreducible and smooth, and its compactification in the toric surface $X_\Delta$ is smooth and Newton nondegenerate.  The logarithmic critical locus of $C$ is smooth, the contour $\mathcal C(\mathscr A_{f_0})$ is smooth, and
$
\mathcal C(\mathscr A_{f_0})=\partial\mathscr A_{f_0}.
$
Moreover, the exponent differences of $\supp(f_0)$ generate $\Z^2$, but $C$ is not torus-equivalent to a real curve.  In particular, $C$ is not torus-equivalent to a smooth simple Harnack curve.  Consequently, the primitive contour--boundary rigidity question has a negative answer.
\end{theorem}

\begin{proof}
The differences of the exponents in $\operatorname{supp}(f_0)$ contain
$(1,0)$ and $(0,1)$, so they generate $\Z^2$.  Moreover,
$p-q+q^2=-1549/40000+(313/10000)i\neq0$.  This is the resultant condition
that $P$ and $Q$ have no common zero.  Since $f_0$ is linear in $w$, it follows
that $C$ is irreducible and smooth in $(\C^*)^2$.  The exact inequality
$pq(1-4p)(p-q+q^2)\neq0$ verifies the face nondegeneracy conditions and shows
that the compactification of $C$ in $X_\Delta$ is smooth and Newton
nondegenerate.

We now prove the assertion about the logarithmic critical values.  Write
$z=re^{i\theta}$ and $R=r^2$.  The radial half-line is compactified by
$\rho=R/(1+R)\in[0,1]$, where $\rho=0$ and $\rho=1$ represent respectively
$r=0$ and $r=\infty$.  The angular circle is covered by the rational charts
$t=\tan(\theta/2)$ and $u=1/t$.  In the first chart one has
$e^{i\theta}=(1+it)/(1-it)$, while the omitted point $e^{i\theta}=-1$ is the
origin $u=0$ of the second chart.  Thus every angular calculation below can be
carried out with exact rational coefficients after denominators are cleared.

Set $D(z)=P'(z)Q(z)-Q'(z)P(z)=1-q+2pz+pqz^2$.  Along $C$, the logarithmic
image has the form
$\Log(z,w)=(\log|z|,\log|P(z)|-\log|Q(z)|)$.  On the circle $|z|=r$, its
second coordinate is critical in the angular direction precisely when
$\operatorname{Im}\{-zD(z)/(P(z)Q(z))\}=0$.  Since $\overline z=R/z$ on
$|z|^2=R$, clearing the denominator converts this equation into\\
$\di
G(z,R)=\frac{z^3}{2i}\left[-zD(z)\overline{P(R/z)Q(R/z)}
+\frac{R}{z}\overline{D(R/z)}P(z)Q(z)\right]=0.
$
The conjugation in this formula acts only on the Gaussian-rational
coefficients.  The negative powers cancel, and direct expansion gives
$G\in\Q(i)[z,R]$ with $\deg_zG=6$ and $\deg_RG=3$.  Away from $PQ=0$, this
polynomial equation is exactly the angular critical equation.

An angular critical point is ramified precisely when it is a multiple angular
root.  Since $dz/d\theta=iz\neq0$, this happens precisely when
$G(z,R)=G_z(z,R)=0$.  Exact fraction-free elimination over $\Q(i)$ gives
$\Res_z(G,G_z)=\kappa R^7H(R)$, where $\kappa\in\Q(i)^*$ and
$H\in\Z[R]$ is primitive and square-free of degree sixteen.  Its coefficients,
listed from the constant coefficient to the leading coefficient, are
\noindent\resizebox{\textwidth}{!}{$
\begin{aligned}
(&7207902987052764892578125000000000000000000000000000000000000000000000000,\\
&-2235664883774034423828125000000000000000000000000000000000000000000000000,\\
&323263372826054470938110351562500000000000000000000000000000000000000000,\\
&-28432824266929992262420959472656250000000000000000000000000000000000000,\\
&1713748656929815624576683444976806640625000000000000000000000000000000,\\
&-76258258047778730562797925381469726562500000000000000000000000000000,\\
&2585361257373668968194341952653045654296875000000000000000000000000,\\
&-62272433585494913399186022213108840942382812500000000000000000000,\\
&837773868564030759634867428063503265380859375000000000000000000,\\
&-2116593353209226767609875029366983817138671875000000000000000,\\
&2940337043539779894955964581775835813720703125000000000000,\\
&-2887072107098313538137371140290748532153320312500000000,\\
&2218827526631833128183852845408180209375000000000000,\\
&-1302026814414436234902749981718246887570312500000,\\
&534216124115740454502201243536113141600000000,\\
&-133751833089116553372095394000103147603125,\\
&15582144860896390055738008885032110741).
\end{aligned}
$}
The Euclidean algorithm gives $\gcd(H,H')=1$.  The degrees in the Sturm chain
of $H$ are $16,15,14,13,12,11,10,9,8,7,6,5,4,3,2,1,0$.  The exact signs at
$0^+$ are\\
$(+,-,+,+,-,-,+,+,+,-,-,+,+,-,-,+,+)$, and the signs at $+\infty$ are\\
$(+,+,+,-,+,+,+,-,-,+,-,-,+,-,-,-,+)$.  Both sequences have eight sign
variations.  Sturm's theorem therefore gives $N_H(0,+\infty)=8-8=0$.
Consequently, the angular-ramification system has no solution with $R>0$.
In particular, there is no positive candidate-radius interval and hence no
candidate lift requiring an interval Newton or Krawczyk test.  The empty lift
list is certified by exact root counting, which is stronger here than a
numerical exclusion.

It remains to count the simple angular roots.  At the rational sample radius
$R=1$, substitute $z=(1+it)/(1-it)$ into $G(z,1)$ and clear the factor
$(1-it)^6$.  Up to a nonzero rational constant, the resulting polynomial is
$T(t)=230308t^6-2191520620t^5-172280867t^4-4564066240t^3
-362837758t^2-2392265620t-175926583$.  Its exact Sturm chain has degrees
$6,5,4,3,2,1,0$ and isolates precisely one real root in
$(-74/1000,-73/1000)$ and one in $(9515,9516)$; the Sturm variation outside
these intervals is zero.  The second root is also isolated in the reciprocal
chart by $1/9516<u<1/9515$.  Hence the two rational charts together contain
exactly two angular critical points at $R=1$.

For each $R>0$, the roots of $G(\,\cdot\,,R)$ have the reciprocal-conjugate
symmetry $z\mapsto R/\overline z$.  A root on $|z|^2=R$ is fixed by this
involution.  Its locally unique continuation remains fixed as long as the root
is simple.  Such a root can leave the circle only by colliding with its
reciprocal-conjugate partner, and that collision would give a multiple root and
force $\Res_z(G,G_z)=0$.  Since the resultant is nonzero for every $R>0$, the
number of angular critical points is constant on the positive radial axis,
apart from the zero and pole radii that must be interpreted in toric
coordinates.  The sample count therefore proves that every nonexceptional
circle has exactly two nondegenerate angular critical points.

The exceptional points can be checked exactly.  The zeros of $P$ are
$z_\pm=(-1\pm\sqrt{1-4p})/(2p)$, and the zero of $Q$ is $z_Q=-1/q$, with
$|z_Q|^2=40000/149$.  If $P(z_0)=0$, then $P'(z_0)\neq0$ and $Q(z_0)\neq0$,
and the local expansion
$w=-(P'(z_0)/Q(z_0))(z-z_0)+O((z-z_0)^2)$ shows that $w$ is a local coordinate
with a simple zero.  Thus $\log|w|\to-\infty$ along a smooth downward
tentacle.  At $z_Q$, the reciprocal toric coordinate $v=1/w=-Q/P$ satisfies
$v=-(q/P(z_Q))(z-z_Q)+O((z-z_Q)^2)$.  Hence $v$ has a simple zero and
$\log|w|=-\log|v|\to+\infty$ along a smooth upward tentacle.  No additional
finite critical branch or ramification point occurs at these radii.

Define $Y(s,\theta)=\log|P(e^{s+i\theta})|-\log|Q(e^{s+i\theta})|$.
For every nonexceptional $s$, the function $Y(s,\,\cdot\,)$ has exactly two
nondegenerate critical points, necessarily its global minimum and global
maximum.  Their angles depend smoothly on $s$ by the implicit-function
theorem.  The corresponding critical-value branches are graphs
$s\mapsto(s,Y(s,\theta_\pm(s)))$; their first coordinate has derivative one,
so both branches are immersed and smooth.  They do not intersect because the
minimum is strictly smaller than the maximum.  The preceding toric expansions
prove smoothness at their noncompact ends.

For fixed $s$, the vertical slice of the amoeba is precisely the interval
$[\min_\theta Y(s,\theta),\max_\theta Y(s,\theta)]$.  At a zero or pole radius
this interval becomes a half-line, and its finite endpoint is the remaining
angular critical value.  Thus the two smooth extremal graphs constitute both
the complete logarithmic critical-value set and the complete amoeba boundary.
This proves $\mathcal C(\mathscr A_{f_0})=\partial\mathscr A_{f_0}$ and proves
that this common set is smooth.

Finally, after the coefficients of $1,z,w$ are normalized to be positive real,
torus reality would require the normalized coefficients $p$ and $q$ to be
real.  Both have nonzero imaginary part, so $C$ is not torus-equivalent to a
real curve.  In particular, it is not torus-equivalent to a smooth simple
Harnack curve.  Together with the primitive-support calculation, this proves
all the assertions.
\end{proof}

\begin{figure}[ht]
\centering
\includegraphics[width=0.32\textwidth]{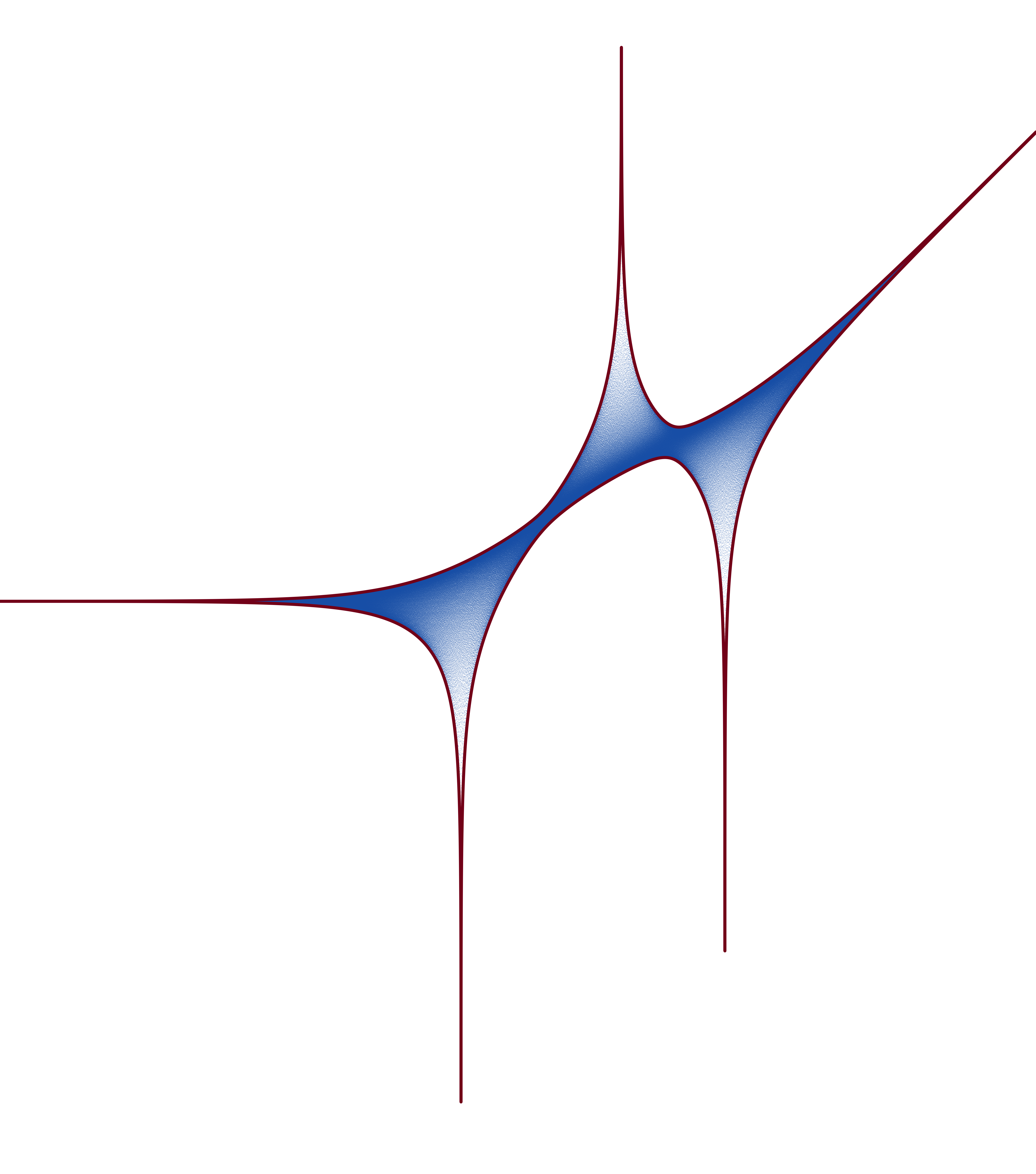}
\caption{Dark-blue amoeba and clear red logarithmic critical values of the curve defined by $f_0(z,w)=0$}
\end{figure} 

%%%%%%%%%%%%%%%%%%%%%%%%%%%%%%%%%%%%%%%%%%%%%%%%%%%%%%%%%%%%%%%%%%%%%%%%%%
%%%%%%%%%%%%%%%%%%%%%%%%%%%%%%%%%%%%%%%%%%%%%%%%%%%%%%%%%%%%%%%%%%%%%%%%%%
%%%%%%%%%%%%%%%%%%%%%%%%%%%%%%%%%%%%%%%%%%%%%%%%%%%%%%%%%%%%%%%%%%%%%%%%%%

\section{Appendix: A Complete Exponent-Lattice Computation}
 
We apply the exponent-lattice criterion to the four-point support
$A=\{(0,0),(4,0),(1,2),(0,2)\}$.  Its convex hull is the quadrilateral
$\Delta=\operatorname{conv}\{(0,0),(4,0),(1,2),(0,2)\}$.  Thus the example
has the same Newton polygon that appeared in the preceding discussion, but
only the four vertex monomials are retained.

Let
$f(z,w)=c_{00}+c_{40}z^4+c_{12}zw^2+c_{02}w^2$, where every displayed
coefficient is nonzero.  Then $\supp(f)=A$ and $\Newt(f)=\Delta$.  For a
fully numerical instance one may take
$f(z,w)=1+2z^4+3zw^2+5w^2$.  The lattice calculation is independent of the
chosen nonzero coefficient values, so we first keep the coefficients symbolic.

Choose the base exponent $a_0=(0,0)$.  The translated nonzero support vectors
are $v_1=(4,0)$, $v_2=(1,2)$ and $v_3=(0,2)$.  Placing them as columns gives
the translated exponent matrix
$$
V=\begin{pmatrix}4&1&0\\0&2&2\end{pmatrix}.
$$
The image of the homomorphism $V:\mathbb Z^3\to\mathbb Z^2$ is the
exponent-difference lattice
$L_f=\mathbb Zv_1+\mathbb Zv_2+\mathbb Zv_3$.

There are three $2\times2$ minors.  Direct calculation gives
$\det(v_1,v_2)=8$, $\det(v_1,v_3)=8$, and
$\det(v_2,v_3)=2$.  Their greatest common divisor is
$\gcd(8,8,2)=2$.  The determinant-gcd criterion therefore gives
$[\mathbb Z^2:L_f]=2$.  In particular, $L_f$ is a proper sublattice and the
exponent differences do not generate all of $\mathbb Z^2$.

The lattice can also be identified directly.  Since
$v_2-v_3=(1,0)$ and $v_3=(0,2)$, the vectors $(1,0)$ and $(0,2)$ belong to
$L_f$.  Conversely, each of $v_1$, $v_2$ and $v_3$ has even second
coordinate, so every element of $L_f$ has even second coordinate.  It follows
that
$$
L_f=\mathbb Z(1,0)\oplus\mathbb Z(0,2)
=\{(m,2n):m,n\in\mathbb Z\}.
$$
The quotient $\mathbb Z^2/L_f$ is isomorphic to $\mathbb Z/2\mathbb Z$ via
the parity of the second coordinate, which confirms the index computation.

We now compute the Smith normal form explicitly.  Starting with $V$, interchange
the first two columns, replace the second column by the second column minus
four times the first, replace the second row by the second row minus twice the
first row, interchange the second and third columns, and finally replace the
third column by the third column plus four times the second.  The successive
matrices are
$$
\begin{aligned}
\begin{pmatrix}4&1&0\\0&2&2\end{pmatrix}
&\longmapsto
\begin{pmatrix}1&4&0\\2&0&2\end{pmatrix}
\longmapsto
\begin{pmatrix}1&0&0\\2&-8&2\end{pmatrix}\\
&\longmapsto
\begin{pmatrix}1&0&0\\0&-8&2\end{pmatrix}
\longmapsto
\begin{pmatrix}1&0&0\\0&2&-8\end{pmatrix}
\longmapsto
\begin{pmatrix}1&0&0\\0&2&0\end{pmatrix}.
\end{aligned}
$$
Every operation is unimodular over $\mathbb Z$.  Thus the final matrix is the
Smith normal form, and its invariant factors are $d_1=1$ and $d_2=2$.  Their
product is $2$, again equal to the index and to the gcd of the maximal minors.

One can record the whole reduction in a single matrix identity.  Define
$$
U=\begin{pmatrix}1&0\\-2&1\end{pmatrix},
\qquad
W=\begin{pmatrix}0&0&1\\1&0&-4\\0&1&4\end{pmatrix}.
$$
Both matrices are unimodular: $\det U=1$ and $\det W=1$.  A direct
multiplication gives
$$
UVW=\begin{pmatrix}1&0&0\\0&2&0\end{pmatrix}.
$$
This identity supplies an exact certificate of the Smith normal form.

The basis $m_1=(1,0)$, $m_2=(0,2)$ of $L_f$ gives the exponent matrix
$M=\operatorname{diag}(1,2)$.  Its determinant is $2$, and the corresponding
finite monomial map is
$\Phi_M:(\mathbb C^*)^2\to(\mathbb C^*)^2$,
$\Phi_M(z,w)=(z,w^2)$.  Its kernel is
$\{(1,1),(1,-1)\}$, so it is a covering of degree two.

To construct the reduced polynomial, write each translated exponent in the
basis $m_1,m_2$.  The identities
$(4,0)=4m_1$, $(1,2)=m_1+m_2$, and $(0,2)=m_2$ show that the corresponding
reduced exponents are $(4,0)$, $(1,1)$ and $(0,1)$.  Hence
$$
h(u,v)=c_{00}+c_{40}u^4+c_{12}uv+c_{02}v.
$$
Substitution gives the exact equality
$$
h\bigl(\Phi_M(z,w)\bigr)
=h(z,w^2)
=c_{00}+c_{40}z^4+c_{12}zw^2+c_{02}w^2
=f(z,w).
$$
Thus $V(f)=\Phi_M^{-1}(V(h))$.  For the numerical coefficients chosen above,
the reduced Laurent polynomial is $h(u,v)=1+2u^4+3uv+5v$, and
$f(z,w)=h(z,w^2)$.

The reduction is maximal.  Indeed, the support of $h$ is
$B=\{(0,0),(4,0),(1,1),(0,1)\}$.  Relative to $(0,0)$, the determinants are
$4$, $4$ and $1$, because
$\det((1,1),(0,1))=1$.  Their gcd is one, so the exponent-difference lattice
of $h$ is all of $\mathbb Z^2$.  Therefore no further finite monomial covering
is forced by the positions of the exponents of $h$.

The Newton polygons behave compatibly with the covering.  If
$\Delta_h=\operatorname{conv}(B)$, then $\Delta=M\Delta_h$.  Since
$|\det M|=2$, Euclidean areas satisfy
$\operatorname{Area}(\Delta)=2\operatorname{Area}(\Delta_h)$.  In the present
case $\operatorname{Area}(\Delta)=5$ and
$\operatorname{Area}(\Delta_h)=5/2$, which provides an additional numerical
check of the determinant.

\begin{proposition}
For $f(z,w)=c_{00}+c_{40}z^4+c_{12}zw^2+c_{02}w^2$ with all coefficients
nonzero, the exponent-difference lattice has index two.  Its Smith normal form
has invariant factors $1$ and $2$, and the polynomial is the pullback of
$h(u,v)=c_{00}+c_{40}u^4+c_{12}uv+c_{02}v$ through the degree-two monomial
covering $(z,w)\mapsto(z,w^2)$.  The reduced polynomial $h$ has primitive
exponent-difference lattice.
\end{proposition}

This example also shows why the actual support matters.  The vertex support
of $\Delta$ lies in the proper affine sublattice with even second coordinate.
If one adds the monomial $w$, whose exponent is $(0,1)$, the enlarged support
immediately generates $\mathbb Z^2$, the determinant gcd becomes one, and the
degree-two covering forced by the exponent positions disappears.

\begin{figure}[ht]
\centering
\includegraphics[width=0.32\textwidth]{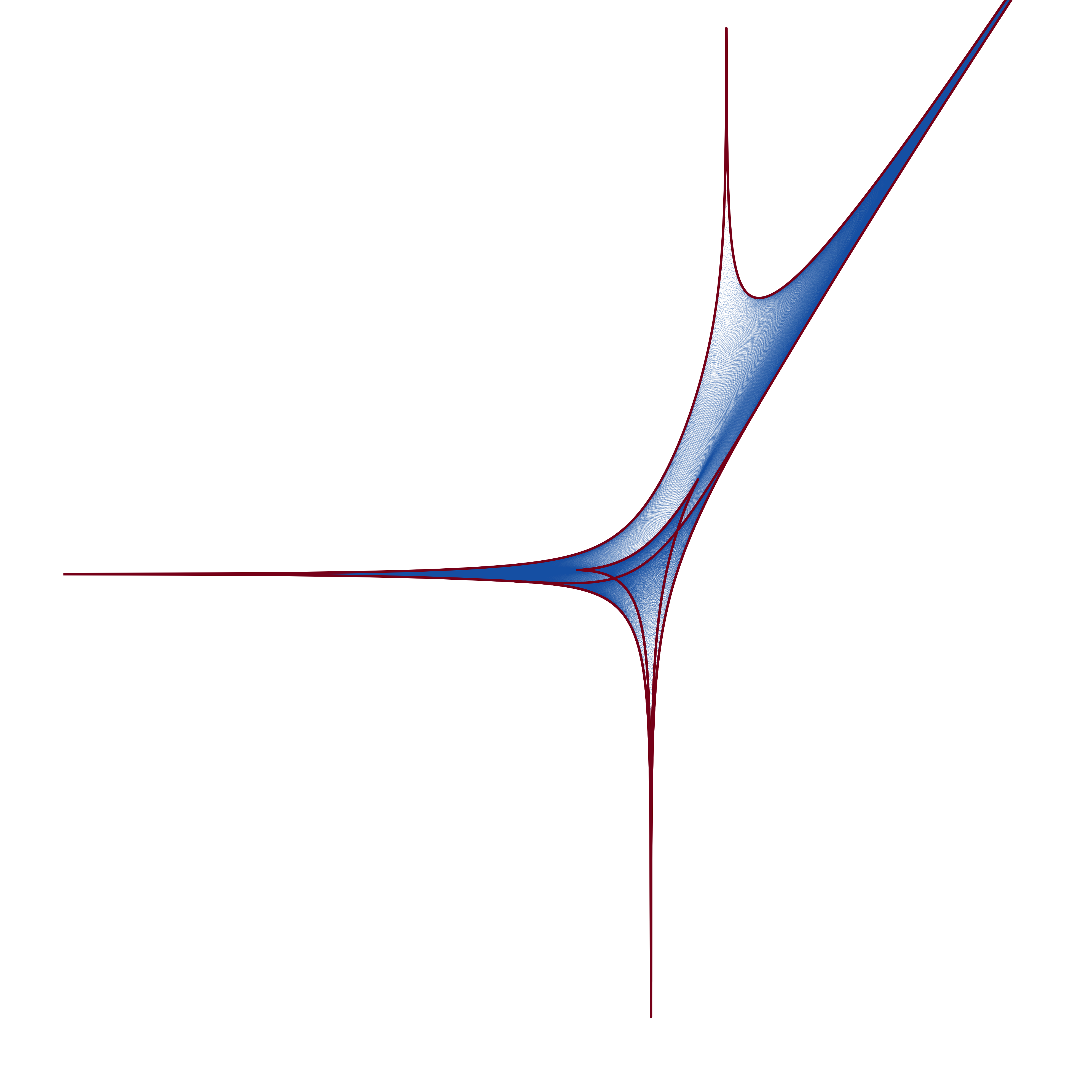}
\caption{Dark-blue amoeba and clear red logarithmic critical values of the curve defined by $f(z,w)=1+2z^4+3zw^2+5w^2$..}
\end{figure}

%%%%%%%%%%%%%%%%%%%%%%%%%%%%%%%%%%%%%%%%%%%%%%%%%%%%%%%%%%%%%%%%%%%%%%%%%

\section{Appendix: A Non-real Curve with Boundary-Only Amoeba Contour}

Consider the rational complex numbers
$p=1/100-i/5000$ and $q=1/20-7i/200$, and define
$$
f(z,w)=1+z+pz^2+w+qzw.
$$
Its support is the complete lattice support of $\Delta$.  In particular, the
support differences contain $(1,0)$ and $(0,1)$ and therefore generate
$\mathbb Z^2$.  The exponent-difference lattice has index one, so the curve is
not a pullback through a finite monomial covering for a support-lattice reason.

The main result is the following.

\begin{theorem}
Let $C=V(f)\subset(\mathbb C^*)^2$ for the polynomial above.  Then $C$ is
irreducible and smooth, its compactification in $X_\Delta$ is smooth and Newton
nondegenerate, and the logarithmic critical locus is smooth.  Moreover,
$\mathcal C(\mathscr A_f)=\partial\mathscr A_f$, and
$\Log|_C$ is at most two-to-one.  Nevertheless, $C$ is not real up to any
torus translation and hence is not torus-equivalent to a smooth simple Harnack
curve.
\end{theorem}

\begin{proof}
Write $P(z)=1+z+pz^2$ and $Q(z)=1+qz$.  The curve is the graph
$w=-P(z)/Q(z)$.  A common zero of $P$ and $Q$ exists exactly when
$p+q^2-q=0$.  For the chosen coefficients one has
$p+q^2-q=(-1549+1252i)/40000\ne0$.  Thus $P$ and $Q$ are relatively prime,
and the polynomial $P+Qw$, which is primitive and linear in $w$, is
irreducible.  Since $f_w=Q$, a torus singularity would force $P=Q=0$, which has
just been excluded.  Hence $C$ is smooth.

The bottom edge polynomial is $P$, whose discriminant is
$1-4p=0.96+0.0008i\ne0$.  The remaining edge polynomials are nonzero monomials
or binomials because $p,q\ne0$.  Every edge truncation is therefore
nondegenerate.  The two-dimensional truncation is $f$ itself and is
nondegenerate because $C$ is smooth in the torus.  The polygon $\Delta$ is a
smooth lattice polygon, so its toric surface is smooth.  It follows that the
compactification of $C$ in $X_\Delta$ is smooth and Newton nondegenerate.

It remains to study the logarithmic map.  Put $z=re^{i\theta}$ and
$Y_r(\theta)=\log|P(re^{i\theta})/Q(re^{i\theta})|$.  The logarithmic map is
$F(r,\theta)=(\log r,Y_r(\theta))$.  Its critical equation is
$Y_r'(\theta)=0$, or, away from zeros and poles,
$\operatorname{Im}(z(P'/P-Q'/Q))=0$.  Multiplication by the positive factor
$|PQ|^2$ gives the cleared real equation
$N_r(\theta)=\operatorname{Im}(A(z)\overline{B(z)})=0$, where
$A(z)=z(P'Q-PQ')=(1-q)z+2pz^2+pqz^3$ and
$B(z)=P(z)Q(z)=1+(1+q)z+(p+q)z^2+pqz^3$.

The exact rational interval certificate proved below shows that $N_r$ has
exactly two simple zeros modulo $2\pi$ for every $r>0$.  On a circle containing
neither a zero of $P$ nor a zero of $Q$, these are precisely the two critical
points of $Y_r$.  The function $Y_r$ is therefore strictly monotone on each of
the two complementary arcs and has one global maximum and one global minimum.
Its image is the interval between them, every interior value has exactly two
preimages, and its two critical values are exactly the endpoints of the
vertical amoeba slice.

If a circle contains a zero of $P$ or $Q$, the corresponding zero of the
cleared equation is a limiting zero at which $Y_r$ tends to $-\infty$ or
$+\infty$.  Relative primeness excludes a simultaneous zero.  Removing that
puncture leaves the same monotonic two-arc description by continuity from
nearby radii.  Every finite logarithmic fiber still has at most two points, and
the finite endpoint of the slice is the image of the remaining critical point.
It follows for all radii that $\Log|_C$ is at most two-to-one and that the
complete finite critical-value set is exactly $\partial\mathscr A_f$.

At every actual critical point, $PQ\ne0$.  Simplicity of the corresponding
zero of $N_r$ gives $\partial N_r/\partial\theta\ne0$, and division by
$|PQ|^2$ gives $\partial Y_r'/\partial\theta\ne0$.  The implicit-function
theorem therefore makes the logarithmic critical locus a smooth real
one-dimensional submanifold of $C$.  This proves all the geometric assertions.

It remains to exclude torus-reality.  The polynomial is already normalized so
that the coefficients of $1,z,w$ are equal to $1$.  Suppose multiplication of
the equation by $\lambda\in\mathbb C^*$ and substitutions
$z\mapsto\alpha z$, $w\mapsto\beta w$ made every coefficient real.  Reality of
the coefficients of $1,z,w$ forces the arguments of $\lambda$, $\alpha$, and
$\beta$ to belong to $\pi\mathbb Z$.  Reality of the remaining two
coefficients would then force $p,q\in\mathbb R$.  This is impossible because
$\operatorname{Im}p=-1/5000$ and $\operatorname{Im}q=-7/200$.  Hence the curve
is not real up to torus translation.  Integral unimodular monomial changes
commute with complex conjugation, so allowing such a torus automorphism cannot
create a real structure that was absent before.  Since a smooth simple Harnack
curve is real, $C$ cannot be torus-equivalent to one.
\end{proof}

Write $N_r(\theta)=C_0(r)+2\operatorname{Re}(C_1(r)e^{i\theta}+
C_2(r)e^{2i\theta}+C_3(r)e^{3i\theta})$.  Direct expansion gives a common
positive factor $r$.  After dividing it out, the four coefficients are

$$
\begin{aligned}
C_0/r={}&\frac7{100}r+\frac{17}{25000}r^3,\\
C_1/r={}&\left(\frac7{400}-\frac{19}{40}i\right)
+\left(\frac{1791}{100000}+\frac{17377}{1000000}i\right)r^2\\
&+\left(\frac{6891}{5000000000}-\frac{4031}{250000000}i\right)r^4,\\
C_2/r={}&\left(-\frac1{5000}-\frac1{100}i\right)r
+\left(-\frac{149}{200000000}-\frac{149}{4000000}i\right)r^3,\\
C_3/r={}&\left(-\frac9{50000}-\frac{493}{2000000}i\right)r^2.
\end{aligned}
$$
No decimal approximation occurs in these expressions.

Put $A_1(r)=2|C_1(r)|$, $B_1(r)=|C_0(r)|+2|C_2(r)|+2|C_3(r)|$, and
$D_1(r)=4|C_2(r)|+6|C_3(r)|$.  The following elementary lemma is the analytic
core of the certificate.

\begin{lemma}
If $B_1(r)<A_1(r)$ and
$D_1(r)^2<A_1(r)^2-B_1(r)^2$, then $N_r$ has exactly two simple zeros modulo
$2\pi$.
\end{lemma}

\begin{proof}
After translating $\theta$, the first harmonic can be written
$A_1\sin\theta$.  Let $G$ be the sum of the constant, second, and third
harmonics.  Then $|G|\leq B_1$ and $|G'|\leq D_1$.  A zero of
$A_1\sin\theta+G(\theta)$ must satisfy
$|\sin\theta|\leq B_1/A_1$.  Hence every zero lies in one of two disjoint
intervals centered at $0$ and $\pi$.  At any zero, the absolute derivative of
the first harmonic is at least $\sqrt{A_1^2-B_1^2}$, which is strictly larger
than $|G'|$.  The full derivative has a fixed nonzero sign at every zero in
each of the two intervals.  The endpoint signs give existence, and the fixed
crossing direction gives uniqueness.  Thus there is one simple zero in each
interval and no other zero.
\end{proof}

To certify these inequalities simultaneously for all $0<r<\infty$, set
$t=r/(1+r)$.  Then $0<t<1$ and $r=t/(1-t)$.  For each coefficient after
division by $r$, define
$\widehat C_k(t)=(1-t)^4(C_k(r)/r)$.  Every $\widehat C_k$ is a complex
polynomial of degree at most four with rational coefficients.  Multiplication
by the common positive factor $(1-t)^4/r$ leaves all homogeneous dominance
inequalities unchanged.  The two infinite logarithmic tails are thereby
included as the endpoints $t=0$ and $t=1$ of one compact interval.

The accompanying verifier evaluates the real and imaginary parts of the
$\widehat C_k$ by exact rational interval Horner arithmetic.  If a complex
interval is $[a,b]+i[c,d]$, it bounds its squared modulus below by the squared
distance of each real interval from zero and above by
$\max(a^2,b^2)+\max(c^2,d^2)$.  Square-root endpoints are enclosed by dyadic
rationals with denominator $2^{100}$, computed using integer square roots.
Thus every lower bound is rounded downward and every upper bound upward; no
binary floating-point operation participates in a certification decision.

On each dyadic interval $I\subset[0,1]$, the verifier computes a lower bound
$A_I$ for $A_1$, upper bounds $B_I,D_I$ for $B_1,D_1$, and accepts $I$ only
when $A_I-B_I>0$ and $A_I^2-B_I^2-D_I^2>0$.  Otherwise it bisects $I$.  The
algorithm terminated with $9165$ adjacent dyadic intervals covering the whole
closed interval $[0,1]$; the maximum subdivision depth was $19$.  The smallest
accepted rational margins were
$$
\begin{aligned}
A_I-B_I&\geq
\frac{30713972841850102093004849}
{198070406285660843983859875840000}>0,\\
A_I^2-B_I^2-D_I^2&\geq
\frac{7971763195456231183593994208654509538545235061079}
{6277101735386680763835789423207666416102355444464034512896000000}>0.
\end{aligned}
$$
Their approximate values, used only for readability, are $1.55065\times10^{-7}$
and $1.26997\times10^{-15}$.  

The strict inequalities at every $t\in[0,1]$ prove the hypotheses of the lemma
for every $r>0$.  This supplies the missing exact statement in the preceding
investigation and upgrades the rational candidate to a certified
counterexample.  It follows that, even for a primitive polygon and a smooth
Newton-nondegenerate curve with smooth critical locus, the equality
$\mathcal C(\mathscr A_f)=\partial\mathscr A_f$ and the at-most-two-to-one
property do not force torus-reality.  Consequently they do not force the curve
to be torus-equivalent to a smooth simple Harnack curve.

There is no conflict with the maximal-area characterization of Mikhalkin and
Rullg\aa rd \cite{MikhalkinRullgard2001}.  Maximal area forces a compatible real
structure in addition to the logarithmic multiplicity condition.  Since the
present curve is not torus-real, its amoeba is not maximal-area.  The example
therefore separates boundary--contour equality and two-sheetedness from
maximal amoeba area.

The next useful computation is the certified amoeba area of this example and a
comparison with $\pi^2\Area(\Delta)$.  One can also derive the open coefficient
region around $(p,q)$ on which the two strict harmonic inequalities persist;
this would produce a full family of nonreal counterexamples rather than one
rational point.

%%%%%%%%%%%%%%%%%%%%%%%%%%%%%%%%%%%%%%%%%%%%%%%%%%%%%%%%%%%%%%%%%%%%%%%%%

\section{Appendix: An Explicit Open Family of Non-real Curves with Boundary-Only Amoeba Contours}

For the primitive non-zonotopal trapezoid
$\Delta=\operatorname{conv}\{(0,0),(2,0),(1,1),(0,1)\}$, an explicit
four-real-dimensional open coefficient neighborhood is constructed in which
every curve is smooth, irreducible, and Newton nondegenerate, while its amoeba
contour is smooth and equals the boundary of the amoeba.  Every member of the
neighborhood has nonreal normalized coefficients and is not torus-equivalent
to a simple Harnack curve.

Consider the normalized five-term family
$f_{p,q}(z,w)=1+z+pz^2+w+qzw$, where $p,q\in\C^{*}$.  Its Newton polygon is
$\Delta=\operatorname{conv}\{(0,0),(2,0),(1,1),(0,1)\}$.  This polygon is a
noncentrally symmetric quadrilateral, has Euclidean area $3/2$, and is not a
zonotope.  Put
$p_0=1/100-i/5000$, $q_0=1/20-7i/200$, and
$\varepsilon=3/50000000$.  Let $\overline U_\varepsilon$ be the closed box
defined by
$$
 |\operatorname{Re}p-1/100|\leq\varepsilon,\qquad
 |\operatorname{Im}p+1/5000|\leq\varepsilon,
 \qquad |\operatorname{Re}q-1/20|\leq\varepsilon,\qquad
 |\operatorname{Im}q+7/200|\leq\varepsilon,
$$
and denote its interior by $U_\varepsilon$.

\begin{theorem}
For every $(p,q)\in\overline U_\varepsilon$, the affine curve
$C_{p,q}=V(f_{p,q})\subset(\C^{*})^2$ is irreducible and smooth, its toric
compactification in $X_\Delta$ is smooth and Newton nondegenerate, and its
logarithmic critical locus is a smooth real one-dimensional manifold.  The
amoeba contour satisfies
$\operatorname{Contour}(\mathscr{A})_{p,q}=\partial\mathscr{A}_{p,q}$ and is a smooth embedded
one-dimensional submanifold of $\mathbb R^2$.  The restriction
$\Log|_{C_{p,q}}$ has at most two points in every fiber.  Moreover, no curve
$C_{p,q}$ in this box is torus-equivalent to a simple Harnack curve.  Hence
$U_\varepsilon$ is an explicit open four-real-dimensional family of nonreal
counterexamples to the assertion that boundary--contour equality forces
torus equivalence to a simple Harnack curve.
\end{theorem}

\begin{proof}
Write $P(z)=1+z+pz^2$ and $Q(z)=1+qz$.  The equation of the curve is
$P(z)+Q(z)w=0$, so away from the zero of $Q$ it is the graph
$w=-P(z)/Q(z)$.  A singular point in $(\C^{*})^2$ would satisfy
$f=f_w=0$.  Since $f_w=Q$, such a point would require $P$ and $Q$ to have a
common zero.  Their resultant, up to a nonzero scalar, is
$p+q^2-q$.  Consequently $p+q^2-q\neq0$ proves smoothness in the torus.
The same expression also excludes a common factor of $P$ and $Q$; because
$f$ is linear in $w$, it therefore proves irreducibility.

The face polynomial on the horizontal edge from $(0,0)$ to $(2,0)$ is
$P(z)$.  Its roots are distinct precisely when $1-4p\neq0$.  The other face
polynomials are binomials with nonzero coefficients, except that their
simultaneous compatibility is again controlled by $p+q^2-q$.  Thus the
conditions $p\neq0$, $q\neq0$, $1-4p\neq0$, and $p+q^2-q\neq0$ imply Newton
nondegeneracy and smoothness of the toric compactification.  Exact interval
evaluation on $\overline U_\varepsilon$ shows that all four quantities avoid
zero.  This part of the verification uses rational endpoints only.

It remains to analyze the logarithmic map.  Put $z=re^{i\theta}$ and define
$Y_r(\theta)=\log|P(z)/Q(z)|$ whenever $P(z)Q(z)\neq0$.  A point of the graph
is critical for $\Log(z,w)=(\log|z|,\log|w|)$ exactly when
$\partial Y_r/\partial\theta=0$.  Clearing the positive denominator gives
the real trigonometric polynomial
$N_r(\theta)=\operatorname{Im}(A(z)\overline{B(z)})$, where
$A(z)=z(P'Q-PQ')=(1-q)z+2pz^2+pqz^3$ and
$B(z)=P(z)Q(z)=1+(1+q)z+(p+q)z^2+pqz^3$.  It has the Fourier expansion
$N_r(\theta)=C_0(r)+2\operatorname{Re}(C_1(r)e^{i\theta}+C_2(r)e^{2i\theta}+C_3(r)e^{3i\theta})$.
The coefficients $C_k(r)$ are explicit polynomials in $r,p,q,\bar p,\bar q$
obtained by expanding $\operatorname{Im}(A\overline B)$.  This form keeps all
coefficient dependence exact and is the expression evaluated by the attached
certificate.

For brevity write
$A_1(r)=2|C_1(r)|$,
$B_1(r)=|C_0(r)|+2|C_2(r)|+2|C_3(r)|$, and
$D_1(r)=4|C_2(r)|+6|C_3(r)|$.
The following elementary criterion is the central point.  If
$B_1<A_1$ and $D_1^2<A_1^2-B_1^2$, then $N_r$ has exactly two zeros modulo
$2\pi$, and both zeros are simple.  Indeed, after rotating $\theta$ by the
argument of $C_1$, the first harmonic is $A_1\cos\theta$.  The sum of the
remaining harmonics has absolute value at most $B_1$.  Hence every zero lies
in one of the two disjoint arcs on which
$|\sin\theta|\geq\sqrt{1-(B_1/A_1)^2}$.  On either arc the derivative of the
first harmonic has magnitude at least $\sqrt{A_1^2-B_1^2}$, whereas the
derivative of the remaining second and third harmonics has magnitude at most
$D_1$.  The second inequality makes the derivative retain a fixed sign on
each arc.  The intermediate value theorem gives one zero on each arc, and
strict monotonicity gives uniqueness and simplicity.

The two simple critical points on every circle $|z|=r$ are respectively the
unique maximum and minimum of $Y_r$.  Therefore their logarithmic images are
the upper and lower endpoints of the vertical slice of the amoeba.  Conversely
every finite endpoint of such a slice is a critical value, so the critical
value set is exactly the topological boundary.  Simplicity of the critical
zeros implies, by the implicit-function theorem, that the critical points
vary smoothly with $r$.  The preceding strict inequalities also keep the two
branches disjoint and prevent a ramification of their logarithmic images.
Consequently the contour is a smooth embedded curve.  Zeros of $P$ and the
zero of $Q$ correspond to ends of the graph rather than additional points of
$(\C^star)^2$; the conclusion across their exceptional radii follows by taking
one-sided limits of the two monotone extremal branches.  Since a fixed value
of $(\log|z|,\log|w|)$ requires a fixed radius $r$ and a solution of
$Y_r(\theta)=y$, the same monotonicity on the two complementary arcs gives at
most two points in each logarithmic fiber.

The two strict inequalities are certified simultaneously for every
$r\in(0,\infty)$ and every $(p,q)\in\overline U_\varepsilon$.  To make the
domain compact, the verification substitutes $t=r/(1+r)$ and homogenizes
each Fourier coefficient to degree four in $t$.  It then evaluates the real
and imaginary parts by outward rational interval arithmetic on a dyadic cover
of $[0,1]$.  Absolute values are enclosed by rational lower and upper square
roots obtained from $110$-bit integer square roots.  On each interval it
proves positive lower bounds for $A_1-B_1$ and
$A_1^2-B_1^2-D_1^2$.  The cover contains $37680$ intervals and reaches depth
$23$.  The smallest certified lower bounds are greater than
$9.261520943710987\cdot10^{-7}$ and
$4.400128109430214\cdot10^{-16}$, respectively.

Finally, the normalization fixes the coefficients of $1$, $z$, and $w$ to
be positive real numbers.  If a diagonal torus rescaling and multiplication
by an overall scalar made the polynomial real, the phases forced by these
three coefficients would already be trivial modulo signs, and therefore the
remaining normalized coefficients $p$ and $q$ would have to be real.  In the
present box one has
$-10003/50000000\leq\operatorname{Im}p\leq-9997/50000000$ and
$-1750003/50000000\leq\operatorname{Im}q\leq-1749997/50000000$, so neither
imaginary part vanishes.  Integral monomial coordinate changes commute with
the real structure and cannot turn a non-torus-real curve into a real one.
Every simple Harnack curve is real, hence none of these curves is
torus-equivalent to a simple Harnack curve.  All assertions hold on the closed
box and therefore, in particular, on its open interior.
\end{proof}

 %%%%%%%%%%%%%%%I%%%%%%%%%%%%%%%%%%%%%%%%%%%%%%%%%%%%%%%%%%%%%%%%%%%%%%%%%%

\section{Appendix: A Boundary--Contour Example with Interior Lattice Points}

The purpose of this appendix is to give a completely explicit plane curve whose
Newton polygon contains interior lattice points, whose affine curve in
$(\C^*)^2$ is smooth, and whose amoeba contour equals the boundary of its
amoeba, although the defining polynomial is nonreal and the curve is not
torus-equivalent to a simple Harnack curve.  The example is reducible: it is a
disjoint union of three smooth affine components.

Let $P(z)=(z-1)(z-2)=z^2-3z+2$ and $Q(z)=z+1$.  Set
$\alpha=(3+4i)/5$ and $\beta=(5+12i)/13$.  These constants satisfy
$|\alpha|=|\beta|=1$, and $1$, $\alpha$, and $\beta$ are pairwise distinct.
Define
$$
 F(z,w)=\bigl(P(z)+Q(z)w\bigr)
 \bigl(P(z)+\alpha Q(z)w\bigr)
 \bigl(P(z)+\beta Q(z)w\bigr).
$$
Expanding only with respect to $w$ gives the useful exact form
$$
 F=P^3+\frac{129+112i}{65}P^2Qw
 +\frac{31+168i}{65}PQ^2w^2
 +\frac{-33+56i}{65}Q^3w^3,
$$
where $P$ and $Q$ mean $P(z)$ and $Q(z)$.  Thus every coefficient is explicit
and belongs to $\mathbb Q(i)$.

\begin{theorem}
Let $C=V(F)\subset(\C^*)^2$.  Then $C$ is a smooth, disconnected affine
curve with three irreducible components.  Its Newton polygon is
$\Delta=\operatorname{conv}\{(0,0),(6,0),(3,3),(0,3)\}$, and the interior of
$\Delta$ contains exactly seven lattice points.  The logarithmic critical
value set of $C$ is a smooth embedded curve and satisfies
$\operatorname{Contour}(\mathscr{A}(C))=\partial\mathscr{A}(C)$.  The exponent differences of
$\operatorname{supp}(F)$ generate $\mathbb Z^2$, so the polynomial does not
factor through a finite monomial covering solely because its support lies in a
proper affine sublattice.  Moreover, $C$ is not torus-equivalent to a simple
Harnack curve.
\end{theorem}

\begin{proof}
Put $u_1=1$, $u_2=\alpha$, and $u_3=\beta$, and write
$H_j(z,w)=P(z)+u_jQ(z)w$.  Each component $C_j=V(H_j)$ is the graph
$w=-u_j^{-1}P(z)/Q(z)$ over the punctured $z$-plane with the zeros of $P$ and
$Q$ removed when necessary.  Since $P$ and $Q$ have no common zero, a point of
$C_j$ cannot satisfy $Q(z)=0$.  Therefore
$\partial H_j/\partial w=u_jQ(z)\neq0$ everywhere on $C_j$, and every
$C_j$ is smooth.

The components are pairwise disjoint in $(\C^*)^2$.  Indeed, if a point
belonged to $C_j\cap C_k$ with $j\neq k$, subtraction of the two equations
would give $(u_j-u_k)Q(z)w=0$.  The factors $u_j-u_k$ and $w$ are nonzero, so
$Q(z)=0$.  Substitution into either equation would then give $P(z)=0$, which
is impossible because $P(-1)=6$.  The union $C=C_1\sqcup C_2\sqcup C_3$ is
therefore smooth as an affine algebraic curve, although it is reducible.

The Newton polygon of every factor $H_j$ is
$\Delta_0=\operatorname{conv}\{(0,0),(2,0),(1,1),(0,1)\}$.  Newton polygons
add under multiplication, hence $\Newt(F)=3\Delta_0$.  This is precisely
$\Delta=\operatorname{conv}\{(0,0),(6,0),(3,3),(0,3)\}$.  Its Euclidean area
is $27/2$.  The lattice lengths of its four edges are $6$, $3$, $3$, and $3$,
so its boundary contains $15$ lattice points when the vertices are counted in
the standard cyclic manner.  Pick's formula gives
$I(\Delta)=27/2-15/2+1=7$.  More concretely, the interior lattice points are
$(1,1),(2,1),(3,1),(4,1),(1,2),(2,2),(3,2)$.

It remains to determine the amoeba and its contour.  Multiplication of the
$w$-coordinate by a complex number of modulus one does not alter its absolute
value.  Hence all three components have exactly the same logarithmic image:
$\mathscr{A}(C_j)=\mathscr{A}(C_1)$.  Their union therefore has
$\mathscr{A}(C)=\mathscr{A}(C_1)$.  The same phase multiplication is a diffeomorphism
of $(\C^*)^2$ commuting with $\Log$, so it carries the logarithmic critical
locus of $C_1$ to that of $C_j$ without changing the critical values.  It is
thus enough to prove $\operatorname{Contour}\mathscr{A}(C_1)=\partial\mathscr{A}(C_1)$ and the smoothness of
this common contour.

On $C_1$ one has $w=-(z-1)(z-2)/(z+1)$.  Fix $r>0$ and put
$z=re^{i\theta}$ and $u=\cos\theta$.  The second logarithmic coordinate is
$Y_r(\theta)=\log|z-1|+\log|z-2|-\log|z+1|$.  As a function of $u$ it is
$$
 \mathcal Y_r(u)=\frac12\log(r^2+1-2ru)
 +\frac12\log(r^2+4-4ru)
 -\frac12\log(r^2+1+2ru).
$$
Whenever the expression is finite, differentiation gives
$\mathcal Y_r'(u)=-r/(r^2+1-2ru)-2r/(r^2+4-4ru)-r/(r^2+1+2ru)$.  Every
denominator is positive, and consequently $\mathcal Y_r'(u)<0$.  Thus
$\mathcal Y_r$ is strictly decreasing on $[-1,1]$.  Since
$u=\cos\theta$, the angular function $Y_r$ has exactly one maximum and one
minimum, occurring at $\theta=\pi$ and $\theta=0$, apart from the radii at
which one of these values is an end of the affine graph.  Those exceptional
radii are handled by the corresponding one-sided limits and introduce no
additional finite critical branch.

For fixed $x=\log r$, the vertical slice of the amoeba is the image of the
continuous angular function $Y_r$.  Its image is the interval between its
unique minimum and unique maximum.  The two endpoints are precisely the
logarithmic critical values on that circle.  It follows for every vertical
slice, and hence globally, that the critical-value set equals the boundary of
the amoeba.  Strict monotonicity on $0<\theta<\pi$ and on
$\pi<\theta<2\pi$ also shows that the logarithmic map has exactly two points
over every interior point of the amoeba and one point over a regular boundary
point.

The two boundary branches have the explicit parametrizations
$c_+(r)=(\log r,\log|(r-1)(r-2)/(r+1)|)$ and
$c_-(r)=(\log r,\log|(-r-1)(-r-2)/(-r+1)|)$ on their natural intervals.
Their first coordinate has derivative $1/r\neq0$.  Hence each finite branch
is an immersion and has no cusp.  The strict maximum--minimum separation
prevents the two branches from meeting at a finite logarithmic value.  The
common contour is therefore a smooth embedded real curve.

The support is primitive.  The term $P^3$ contains consecutive powers of $z$
with nonzero coefficients, so the exponent-difference lattice contains
$(1,0)$.  The constant terms of $P^3$ and $P^2Qw$ are both nonzero, so it also
contains $(0,1)$.  These two vectors generate $\mathbb Z^2$.  Equivalently,
the greatest common divisor of the $2\times2$ minors of a translated exponent
matrix is one.  Thus this polynomial is not forced to be a monomial pullback
by a proper exponent sublattice.

Finally, the curve is not torus-equivalent to a simple Harnack curve.  The
most immediate reason is that $C$ has three connected irreducible components,
whereas a simple Harnack curve in the usual sense is irreducible.  There is
also a coefficient-phase obstruction to torus reality.  The coefficients of
$P^3$ contain several consecutive nonzero real powers of $z$; after a diagonal
torus rescaling and multiplication by an overall scalar, their simultaneous
reality forces the phase of the $z$-rescaling and the overall phase to be real
modulo signs.  If the phase of the $w$-rescaling is denoted by $\mu$, reality
of the coefficients of $w$ and $w^2$ would force
$B/A^2\in\mathbb R$, where $A=(129+112i)/65$ and $B=(31+168i)/65$.
However, direct calculation gives
$\operatorname{Im}(B/A^2)\neq0$; after clearing the positive real
denominators, its numerator is
$168\cdot4097-31\cdot28896=-207480$.  Hence no torus rescaling makes all
coefficients real.  This completes the proof.
\end{proof}

\begin{figure}[ht]
\centering
\includegraphics[width=0.22\textwidth]{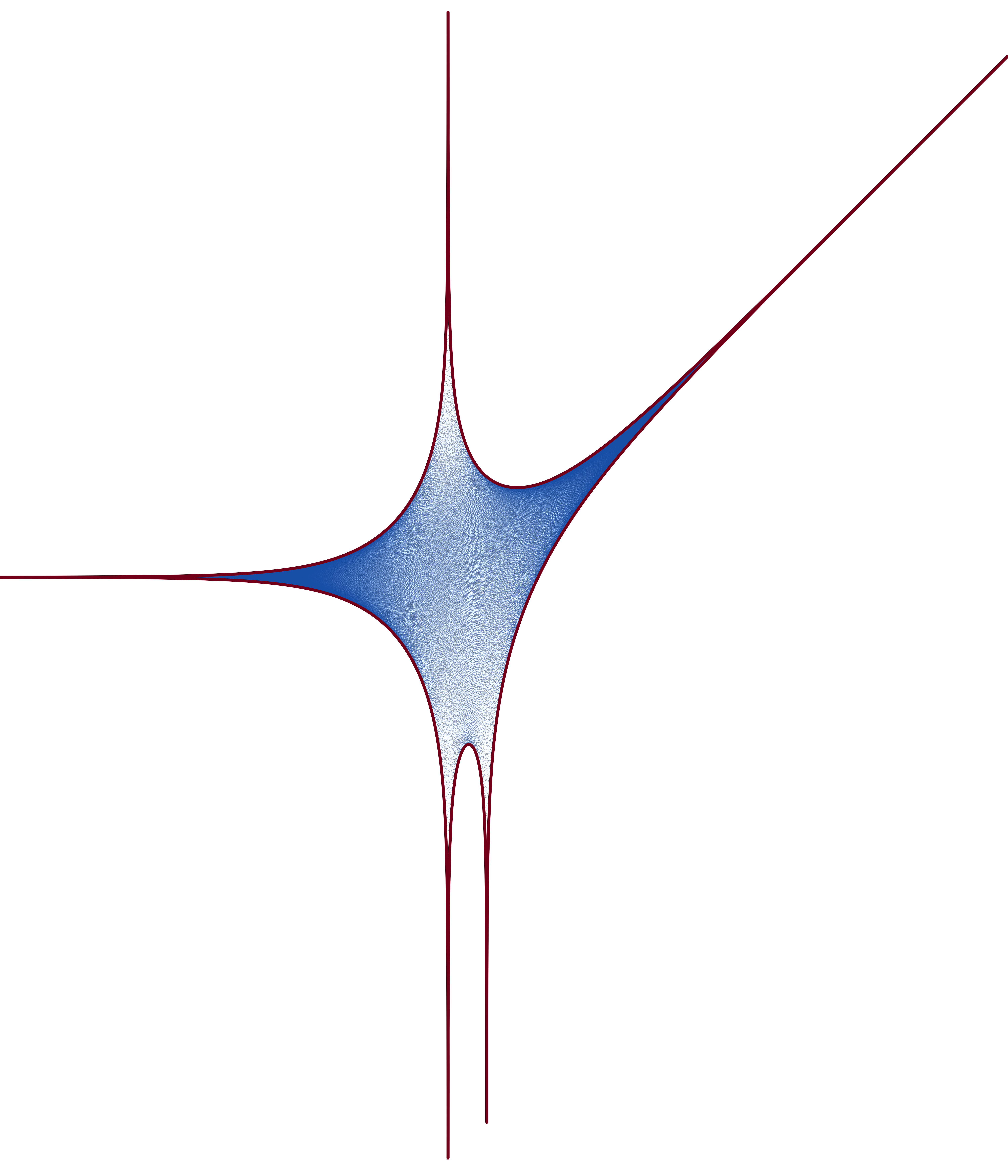}
\caption{Dark-blue amoeba and clear red logarithmic critical values of the curve defined by $F(z,w)=0$}
\end{figure}

\end{document}